\documentclass{amsart}
    \usepackage{amsmath}
    \usepackage{amssymb,amsfonts}
    \usepackage[all,arc]{xy}
    \usepackage{caption}
    \usepackage{enumerate}
    \usepackage{mathrsfs}
    \usepackage{amsthm}
    \usepackage[shortlabels]{enumitem}
    \usepackage{extarrows}
    \usepackage{verbatim}
    \usepackage{dsfont}
    \usepackage{tikz-cd}
    \usepackage{tensor}
    \usepackage{mathtools}
    \usepackage{upgreek}
    \usepackage{xcolor}
    \usepackage{extarrows}
    \usepackage{mathtools}
    \usepackage{tikz}
    \usepackage{pgfplots}
    \usetikzlibrary{positioning}
    \usetikzlibrary{calc}
    \usepackage{ stmaryrd }

    \newtheorem{thm}{Theorem}[section]
    \newtheorem{cor}[thm]{Corollary}
    \newtheorem{prop}[thm]{Proposition}
    \newtheorem{lem}[thm]{Lemma}

    \newtheorem{lemdef}[thm]{Lemma/Definition}
    
    \theoremstyle{definition}
    \newtheorem{defn}[thm]{Definition}

    \newtheorem{exmp}[thm]{Example}

    \theoremstyle{remark}
    
    \newtheorem{rem}[thm]{Remark}

    \newcommand{\Z}{\mathbb{Z}}
    \newcommand{\R}{\mathbb{R}}
    \newcommand{\C}{\mathbb{C}}
    
    \newcommand{\Ecal}{\mathcal{E}}

    \newcommand{\Hcal}{\mathcal{H}}
    \newcommand{\Ical}{\mathcal{I}}
    \newcommand{\Lcal}{\mathcal{L}}
    \newcommand{\Ocal}{\mathcal{O}}

    \newcommand{\Wcal}{\mathcal{W}}
    \newcommand{\Xcal}{\mathcal{X}}
    \newcommand{\Ycal}{\mathcal{Y}}

    \newcommand{\Cbb}{\mathbb{C}}
    \newcommand{\Pbb}{\mathbb{P}}
    \newcommand{\Qbb}{\mathbb{Q}}

    \newcommand{\abf}{\mathbf{a}}
    \newcommand{\bbf}{\mathbf{b}}
    \newcommand{\ibf}{\mathbf{i}}
    \newcommand{\sbf}{\mathbf{s}}
    \newcommand{\vbf}{\mathbf{v}}
        \newcommand{\wbf}{\mathbf{w}}

    \newcommand{\Bl}{\mathrm{Bl}}

    \newcommand{\Coh}{\mathrm{Coh}}

    \newcommand{\NS}{\mathrm{NS}}

    \newcommand{\Supp}{\mathrm{Supp}}
    \newcommand{\supp}{\mathrm{supp}}
    
    \newcommand{\RHom}{R\mathcal{H}om}
    \newcommand{\Ext}{\mathrm{Ext}}

    \newcommand{\ch}{\mathrm{ch}}
    \newcommand{\Pic}{\mathrm{Pic}}
    \newcommand{\half}{\frac{1}{2}}
    
    \newcommand{\hilb}[1]{S^{[#1]}}
    
    \newcommand{\Mov}{\mathrm{Mov}}
    \newcommand{\Halg}{H^*_{\mathrm{alg}}}
    \newcommand{\Db}{\mathrm{D}^b}
    \newcommand{\Stab}{\mathrm{Stab}}

    \makeatletter
    
    \newcommand*{\rom}[1]{\expandafter\@slowromancap\romannumeral #1@}
    
    \makeatletter
    
    \let\c@equation\c@thm
    \makeatother
    \numberwithin{equation}{section}

\title[Birational geometry of Beauville-Mukai systems II]{Birational geometry of Beauville-Mukai systems II: general theory in low ranks}
\author{Xuqiang Qin and Justin Sawon}
\address{Department of Mathematics, University of North Carolina, Chapel Hill, NC 27599--3250, USA}
\email{qinx@unc.edu, sawon@email.unc.edu}

\date{July 2022}

\pgfplotsset{compat=1.14} 

\begin{document}

\begin{abstract}
Via wall-crossing, we study the birational geometry of Beauville-Mukai systems on K3 surfaces with Picard rank one. We show that there is a class of walls which are always present in the movable cones of Beauville-Mukai systems. We give a complete description of the birational geometry of rank two Beauville-Mukai systems when the genus of the surface is small. 
\end{abstract}

\maketitle

\tableofcontents

\section{Introduction}
Let $S$ be a smooth projective K3 surface with $\Pic(S)=\Z\cdot H$ and $H^2=2d$. By a \emph{Beauville-Mukai system} we mean a moduli space of $H$-stable torsion sheaves on $S$ of the form $M_H(0,n,-1)$ for $n\geq 1$. We will call $n$ the rank of the Beauville-Mukai system. There is a natural morphism
\begin{align*}
    M_H(0,n,-1)\to |nH|
\end{align*}
to the complete linear system $|nH|$ given by taking the Fitting support. Mukai \cite{Muk84} showed that $M_H(0,n,-1)$ has a holomorphic symplectic structure. Beauville \cite{Bea91} showed that the above morphism is a Lagrangian fibration, i.e., a fibration by complex tori that are Lagrangian with respect to the holomorphic symplectic form. (Beauville assumed that $n=1$ for simplicity, though his argument works in general; the result also follows from a later theorem of Matsushita \cite{Mat99}.)
 
It is known that $M_H(0,n,-1)$ is birational (but not isomorphic) to the Hilbert scheme of $dn^2+1$ points on $S$; see the discussion and references in the introduction of \cite{QS22a}. In this paper, we use the techniques of wall-crossing from Bridgeland stability to study the birational geometry of these schemes.

The simplest case is when $n=d=1$. Then $\pi:S\to \Pbb^2$ is a double cover ramified at a smooth sextic curve. Mukai \cite{Muk84} noticed that the birational map 
\begin{align*}
    \hilb{2}\dashrightarrow M(0,1,-1)
\end{align*}
is a simple flop at the locus $\{\pi^{-1}(p)\;|\; p\in\Pbb^2\}\cong \Pbb^2$. Moreover, the exceptional locus in $M(0,1,-1)$ is the Brill-Noether locus $\{\Ecal\;|\; h^0(\Ecal)\geq 1\}$. 

By wall-crossing, Bayer \cite{Bay18} showed that if we keep $n=1$ but allow $d\geq 1$, then there is a stratified Mukai flop of $M(0,1,-1)$ at the Brill-Noether locus, resulting in a hyperk\"ahler manifold birational to $M(0,1,-1)$. Combining with a result in \cite{BM14a}, we can show that the resulting manifold is isomorphic to $\hilb{dn^2+1}$.
\begin{thm}\cite[Proposition 10.3]{BM14a}\cite{Bay18}
The stratified Mukai flop of $M(0,1,-1)$ at the Brill-Noether locus $\{\Ecal\mid h^0(\Ecal)\geq 1\}$ gives rise to a birational map
\begin{align*}
    M(0,1,-1)\dashrightarrow \hilb{dn^2+1}.
\end{align*}
Moreover, these two schemes are the only (projective) hyperk\"ahler manifolds in their birational class.
\end{thm}
The last statement follows from the classical description of the wall and chamber structure of the movable cone of a hyperk\"ahler manifold (see Section~\ref{pre-bm} for a review). In particular, there is a unique flopping wall in the movable cone, which we will call the \emph{Brill-Noether wall}.

In another direction, one can keep $d=1$ and consider higher rank $n$. Hellmann \cite{Hel20} studied the case when $n=2$ and showed that the birational map between $\hilb{5}$ and $M(0,2,-1)$ can be decomposed into four flops between hyperk\"ahler manifolds. The authors \cite{QS22a} studied the case when $n=3$ in a similar manner and showed that  the birational map between $\hilb{10}$ and $M(0,3,-1)$ can be decomposed into ten flops between hyperk\"ahler manifolds. The main strategy used in these two papers is the following: similarly to the rank one case, there is a Brill-Noether wall in the middle of the movable cone. Instead of trying to find a sequence of birational maps all the way from $\hilb{n^2+1}$ to $M(0,n,-1)$, or vice versa, we start at these two moduli spaces separately and meet at the Brill-Noether wall. One advantage of this strategy is that the dimensions of the exceptional loci starting from either side to the Brill-Noether wall are (roughly) increasing, making the interpretation of their geometric meaning much easier.

Inspired by the aforementioned work, we consider the case of general $n$ and $d$ in this paper. We introduce the notion of the \emph{rank} of a wall for walls on either side of the Brill-Noether wall (Definition~\ref{RankDef}), which is roughly the minimal rank of a Mukai vector defining the wall in the space of Bridgeland stability conditions. We show that there is always a collection of rank one walls. Moreover, we give a complete list of rank one walls and show that among them are the walls which are closest to the boundary of the movable cone. 

We apply the theory to Beauville-Mukai systems with rank $n=2$. We derive an algorithm to find all of the walls (Propositions~\ref{AlgS} and \ref{AlgM}). In contrast to earlier work~\cite{BM14b,Cat19}, our algorithm determines all of the Mukai vectors $\abf=(a,b,c)$ that define walls without the need to solve Pell's equation. We also show that when the degree is relatively small, the rank one walls dominate the birational geometry of $M(0,2,-1)$ and the sporadic higher rank walls behave nicely. In turn, we can describe the exceptional locus for each birational map and give a complete description of the birational geometry of $M(0,2,-1)$. In particular, we obtain the following result.
\begin{thm}
Let $N_d$ be the number of non-isomorphic projective hyperk\"ahler manifolds birational to $M(0,2,-1)$ when the degree of $S$ is $2d$. For $d\leq 6$, $N_d$ is given in the table below.
\begin{center}
\begin{tabular}{ |c|c|c|c|c|c|c| } 
 \hline
$d$&$1$&$2$&$3$&$4$&$5$&$6$\\
\hline
$N_d$&$5$&$7$&$10$&$12$&$15$&$17$\\
 \hline
\end{tabular}
\end{center}
\end{thm}

This paper is organized as follows. Section 2 contains some preliminaries. Section 3 reviews the (complete) birational geometry of $M(0,1,-1)$. In Section 4 we study the `rank one' walls of $M(0,n,-1)$ and uses them to motivate the notion of the rank of a wall. Section 5 is a brief excursion to eliminate totally semistable walls on our paths of wall-crossing. In Section 6 we study the walls of $M(0,2,-1)$ using results from Section 4. In Section 7 we analyze the exceptional loci of rank one walls. Finally, Section 8 combines the results from the previous two sections to give a complete description of the birational geometry of $M(0,2,-1)$ when the degree of the K3 surface is small.

\subsection*{Relations to previous work}
Wall-crossing in Bridgeland stability conditions has been used to study the birational geometry of moduli spaces on various surfaces; see \cite{ABCH13} and \cite{LZ18} for the program on the projective plane $\Pbb^2$, \cite{BC13} on Hirzebruch and del Pezzo surfaces, \cite{MM13} on abelian surfaces, and \cite{Nue16} and \cite{NY20} on Enriques surfaces.

On K3 surfaces our main techniques come from the seminal work of Bayer and Macr\`i \cite{BM14a,BM14b}, in which they studied the birational geometry of moduli spaces of stable sheaves on K3 surfaces via wall-crossing. In particular, \cite[Theorem 5.7]{BM14b} provided a numerical criterion for finding all the walls. However, the application of this criterion to Hilbert schemes of points requires solving Pell's equations. Our analysis and algorithm simplifies this process to solving elementary inequalities, thus avoid the difficulty of solving Pell's equations. 

The idea of studying the birational geometry of $M(0,n,-1)$ by starting at $\hilb{dn^2+1}$ and $M(0,n,-1)$ separately and converging at the Brill-Noether wall originated in \cite{Hel20}, where the case $n=2$ and $d=1$ was treated. This was generalized to the $n=3$ and $d=1$ case by the authors \cite{QS22a}. In the present paper we apply this idea to general $n$ and $d$. Moreover, all of the walls in \cite{Hel20} and \cite{QS22a} will turn out to be rank one walls, which are thoroughly studied in Section 4.

In general, the Hilbert scheme of points $\hilb{m}$ is birational to a Lagrangian fibration if and only if $d(m-1)$ is a perfect square (see \cite{BM14b}). Our choice of $m=dn^2+1$ is a special case. Other choices of $m$ have been considered by \cite{Marku06}, \cite{Saw07}, and \cite{BM14a}. In the third installment of this series \cite{QS22c}, we will use similar wall-crossing techniques to show that the Hilbert scheme itself admits a Lagrangian fibration when $d$ is large relative to $m$ (and $d(m-1)$ is a perfect square).

\subsection*{Acknowledgements} The authors would like to thank Emanuele Macr\`i for helpful discussions. The second author gratefully acknowledges support from the NSF, grant number DMS-1555206.

\subsection*{Notation and convention}
Throughout this paper, $(S,H)$ will denote a smooth projective polarized K3 surface over the complex numbers. We assume that $S$ has Picard rank one, $\Pic(S)=\Z\cdot H$, and we let $H^2=2d$. 

We use $\Halg(S,\Z)$ to denote the algebraic cohomology group $H^0(S,\Z)\oplus \NS(S)\oplus H^4(S,\Z)$ of $S$. We use $\Db(S)$ to denote the bounded derived category of coherent sheaves on $S$.

Given $\vbf\in\Halg(S,\Z)$, by $M_H(\vbf)$ and $M_H^{st}(\vbf)$ we mean the moduli spaces of $H$-semistable and $H$-stable sheaves, respectively, with Mukai vector $\vbf$. We will often drop the subscript when no confusion will be caused. Given a Bridgeland stability condition $\sigma$, by $M_\sigma(\vbf)$ and $M^{st}_\sigma(\vbf)$ we mean the moduli spaces of $\sigma$-semistable and $\sigma$-stable objects in $\Db(S)$, respectively, with Mukai vector $\vbf$.

By a hyperk\"ahler manifold we mean a projective hyperk\"ahler manifold.\

\section{Preliminaries}
\subsection{Beauville-Mukai systems}\label{pre-bm}
Let $S$ be a polarized K3 surface as above. The term \emph{Beauville-Mukai system} is used to denote the moduli space $M(0,n,s)$ of $H$-Gieseker semistable sheaves with Mukai vector $(0,n,s)$, where $n\geq 1$. We call $n$ the \emph{rank} of the system. The Beauville-Mukai system $M(0,n,s)$ is smooth when $\mathrm{gcd}(n,s)=1$. In this paper we focus on the case when $s=-1$. Henceforth, by \emph{Beauville-Mukai system of rank $n$} we will mean the moduli space $M(0,n,-1)$. 

For $\Lcal\in \Pic(S)$, let $T_{\Lcal}$ denote the spherical twist by $\Lcal$, which is an autoequivalence of the derived category $\Db(S)$. For $n\geq 1$, we use $\Phi_n$ to denote the autoequivalence
\begin{align*}
  \Phi_n: \Db(S)\to\Db(S)
\end{align*}
given by composing the spherical twist $T_{\Ocal_S(-n)}$ with tensoring by $\Ocal_S(n)$. Then $\Phi_n$ induces a birational map
\begin{eqnarray*}
\Phi_n:\hilb{dn^2+1} & \dashrightarrow & M(0,n,-1), \\
\xi & \mapsto & \Ocal_D(-\xi)\otimes \Ocal_S(n)|_D,
\end{eqnarray*}
where for generic $\xi$, $D\in|nH|$ is the unique smooth curve in $|nH|$ containing $\xi$.

We will also use $\Phi_n$ to denote the induced automorphism of $\Halg(S,\Z)$, which makes the following square commute:
\[
\begin{tikzcd}
\Db(S) \arrow{r}{\Phi_n}\arrow{d}{\vbf} & \Db(S) \arrow{d}{\vbf} \\
\Halg(S,\Z) \arrow{r}{\Phi_n} & \Halg(S,\Z) \
\end{tikzcd}
\]
It is straightforward to check that $\Phi_n:\Halg(S,\Z)\to \Halg(S,\Z) $ is given by the matrix:
\[
\begin{pmatrix}
 -dn^2&-2dn&-1\\
 n&1&0\\
 -1&0&0
\end{pmatrix}
\]

By the assumption that $\Pic(S)=\Z\cdot H$, the N{\'e}ron-Severi group $\NS(\hilb{dn^2+1})$ has rank two with a basis given by
\begin{align*}
    \tilde{H}=\theta(0,-1,0) \qquad\mbox{and}\qquad B=\theta(-1,0,-dn^2),
\end{align*}
where $\theta$ is the Mukai homomorphism. By \cite[Proposition 13.1(a)]{BM14b}, the movable cone of $\hilb{dn^2+1}$ is given by
\begin{align*}
    \Mov(\hilb{dn^2+1})=\left\langle \tilde{H}, \tilde{H}-\frac{1}{n}B\right\rangle.
\end{align*}

A smooth moduli space of sheaves on $S$ has the structure of a hyperk\"ahler manifold. Recall that birational hyperk\"ahler manifolds have the same movable cone \cite[Proposition 3.13]{Deb20}. By \cite[Theorem 7]{HT19}, the movable cone of a hyperk\"ahler manifold $M$ admits a locally polyhedral chamber decomposition, whose chambers correspond to hyperk\"ahler manifolds birational to $M$. We will refer to these hyperk\"ahler manifolds birational to $M$ as \emph{birational models of $M$}.  Thus to understand the birational geometry of the Beauville-Mukai system $M(0,n,-1)$, it suffices to understand the chamber decomposition of $\Mov(\hilb{dn^2+1})$ and the birational maps corresponding to the walls. Since $\tilde{H}$ induces the Hilbert-Chow morphism on $\hilb{dn^2+1}$, the only chamber with $\tilde{H}$ on its boundary is the chamber for $\hilb{dn^2+1}$. On the other hand, $\tilde{H}-\frac{1}{n}B$ induces a Lagrangian fibration on $M(0,n,-1)$ (and thus a rational Lagrangian fibration on $\hilb{dn^2+1}$), and therefore the only chamber with $\tilde{H}-\frac{1}{n}B$ on its boundary is the one for $M(0,n,-1)$.

By the description of $\Mov(\hilb{dn^2+1})$, a wall must be given by a ray (starting at $0$) through $\tilde{H}-\Gamma B$ for some $0<\Gamma<\frac{1}{n}$. We will use the Bridgeland stability techniques developed in \cite{BM14b} to find the walls and analyze their induced birational maps.

\subsection{Stability conditions on  K3 surfaces}\label{pre-stab}
In this section we give a brief review of (Bridgeland) stability conditions on K3 surfaces. A more detailed version can be found in \cite[Section 2]{QS22a}. Let $X$ be a K3 surface. Its \emph{algebraic cohomology group} is 
\begin{align*}
    \Halg(X,\Z):=H^0(X,\Z)\oplus \NS(X)\oplus H^4(X,\Z).
\end{align*}
For $E\in K(\mathrm{D}^b(X))$, its \emph{Mukai vector} is given by
\begin{eqnarray*}
\vbf:K(\mathrm{D}^b(X)) & \to & \Halg(X,\Z), \\
    E & \mapsto & \vbf(E):=\ch(E)\sqrt{\mathrm{td(X)}}.
\end{eqnarray*}
The \emph{Mukai pairing} $(-,-)$ on $\Halg(X,\Z)$ is defined by
\begin{align*}
    ((r,c_1,s),(r',c_1',s')):=c_1 c_1'-rs'-r's\in \Z.
\end{align*}
By Riemann-Roch, $\chi(F',F)=-(\vbf(F'),\vbf(F))$ for any $F,F'\in \mathrm{D}^b(X)$. The group $\Halg(X,\Z)$ endowed with the Mukai pairing is called the \emph{algebraic Mukai lattice} of $X$.\\

We refer the reader to \cite{Bri07,Bri08} for the definitions of slicings, hearts, general stability conditions, and the complex manifold structure on the space of stability conditions. We note that, for the rest of this paper, all stability conditions on any K3 surface $X$ will be with respect to the lattice $\Halg(X,\Z)$.

Let $\beta$ and $\omega\in \mathrm{NS}(X)_\Qbb$ with $\omega$ ample. By tilting the category of coherent sheaves on $X$ based on the Mumford slope, one can define a stability condition $\sigma_{\beta,\omega}$ satisfying the condition that the skyscraper sheaves $\Ocal_x$ are stable for all $x\in X$ (here $\beta$ and $\omega$ must satisfy certain numerical conditions: see Section~\ref{pre-hilb}). Its central charge is given by
\begin{eqnarray*}
Z_{\beta,\omega}:K_{\mathrm{num}}(X) & \to & \C, \\
F & \mapsto & \left(e^{\beta+\sqrt{-1}\omega},\vbf(F)\right).
\end{eqnarray*}

By \cite[Lemma 8.2]{Bri07} the group $\mathrm{Aut}(\mathrm{D}^b(X))$ acts on the left on the space of stability conditions $\mathrm{Stab}(X)$ by $\Psi(\mathcal{P},Z)=(\Psi(\mathcal{P}), Z\circ\Psi^{-1})$, where $\Psi\in \mathrm{Aut}(\mathrm{D}^b(X))$, $\mathcal{P}$ is the slicing, and $Z$ is the central charge of the stability condition. The universal cover $\widetilde{\mathrm{GL}}_2^+(\R)$ of the group of $2\times2$ matrices with real coefficients and positive determinants acts on the right on $\mathrm{Stab}(X)$.
Stability conditions in the $\widetilde{\mathrm{GL}}_2^+(\R)$ orbit of $\sigma_{\beta,\omega}$ for $\beta,\omega\in \mathrm{NS}(X)_\Qbb$ with $\omega$ ample are called \emph{geometric stability conditions}.  The connected component in $\mathrm{Stab}(X)$ containing all geometric stability conditions is denoted by $\mathrm{Stab}^\dagger(X)$. See \cite[Section 8]{Bri08} for its properties. \\

Fixing a Mukai vector $\mathbf{v}\in H^*_{\mathrm{alg}}(X,\Z)$, there exists a locally finite set of walls (real codimension one submanifolds with boundary) in $\mathrm{Stab}^\dagger(X)$, determined solely by $\mathbf{v}$, such that the set of $\sigma$-(semi)stable objects does not change within chambers. A stability condition $\sigma$ is said to be \emph{generic} with respect to $\mathbf{v}$ if it does not lie on a wall for $\mathbf{v}$.

\begin{thm}\cite[Theorem 1.3(a)]{BM14a}
Let $\sigma\in \mathrm{Stab}^\dagger(X)$ be generic with respect to $\mathbf{v}$. There exists a coarse moduli space $M_\sigma(\mathbf{v})$ parametrizing $\sigma$-semistable objects with Mukai vector $\vbf$. Moreover, $M_\sigma(\vbf)$ is a normal irreducible projective variety.

If $\vbf$ is primitive, then $M_\sigma(\vbf)=M^{st}_\sigma(\vbf)$ is a projective hyperk\"ahler manifold and the Mukai homomorphism induces an isomorphism
\begin{align*}
    \theta_\vbf:
    \begin{cases}
        \mbox{   }\vbf^\perp\xrightarrow{\sim}\mathrm{NS}(M_\sigma(\vbf))\qquad &\mbox{ if }\vbf^2>0, \\
        \vbf^\perp/\vbf\xrightarrow{\sim}\mathrm{NS}(M_\sigma(\vbf))\qquad &\mbox{ if }\vbf^2=0.
    \end{cases}
\end{align*}
\end{thm}
Crossing a wall induces a birational map on the moduli space. A wall is either \emph{a divisorial wall, a flopping wall, or a fake wall}, depending on the type of birational map it induces; \cite[Section 5]{BM14b} provides numerical criterion for finding and classifying the walls.

We recall the comparison between Gieseker and Bridgeland stability for large $\omega$.
\begin{thm}\cite[Proposition 14.2]{Bri08}
    Fix $\vbf\in \Halg(X,\Z)$ and $s_0\in\R$. Let $H\in \mathrm{NS}(X)$ be ample with $\mu_{s_0H,H}(\vbf)>0$. Then $M_{\sigma_{s_0H,tH}}(\vbf)=M_H(\vbf)$ for $t\gg0$.
\end{thm}

The following duality result will be used in Section 6.
\begin{prop}\cite[Proposition 2.11]{BM14b}\label{dual}
    The stability conditions $\sigma_{\beta,\omega}$ and $\sigma_{-\beta,\omega}$ are dual to each other: an object $E\in \mathrm{D}^b(X)$ is $\sigma_{\beta,\omega}$-(semi)stable if and only if $\RHom(E,\Ocal_X)[1]$ is $\sigma_{-\beta,\omega}$-(semi)stable.
\end{prop}

\subsection{Walls for $\hilb{dn^2+1}$ and the Beauville-Mukai system}\label{pre-hilb}

Let $W$ be a wall for $\vbf:=(1,0,-dn^2)$ in $\Stab^\dagger(S)$. By \cite[Proposition 5.1]{BM14b}, one can associate to it a primitive hyperbolic rank two sublattice $\mathcal{H}_W$ of $\Halg(S,\Z)$ containing $\vbf$. If $W$ is a flopping wall, then the wall corresponds to a nef divisor $\tilde{H}-\Gamma B$ for some $\Gamma\in\mathbb{Q}$ satisfying $0<\Gamma<\frac{1}{n}$. The rational number $\Gamma$ is determined by requiring that $\tilde{H}-\Gamma B\in \vbf^\perp\cap\abf^\perp$ for any $\abf\in\mathcal{H}_W$ which is not a multiple of $\vbf$. Conversely, any wall whose $\Gamma$ computed as above satisfies $0<\Gamma<\frac{1}{n}$ is a flopping wall.

For $x\in \R$ and $y\in \R_{>0}$, we use $\sigma_{x,y}$ to denote $\sigma_{xH,yH}$ as defined in Section~\ref{pre-stab}. By \cite[Lemma 6.2]{Bri08}, we obtain a set of (geometric) stability conditions parametrized by an open subset of the upper half plane.
\begin{lem}\label{isStab}
For any $x\in\R$ and $y>\frac{1}{\sqrt{d}}$, $\sigma_{x,y}$ is a stability condition on $S$.
\end{lem}
\begin{rem}\label{RestrictionStab}
The restriction $y>\frac{1}{\sqrt{d}}$ is sufficient but not necessary for $\sigma_{x,y}$ to be a stability condition. A more precise requirement is that $Z_{xH,yH}(E)\not\in\R_{\leq0}$ for any spherical sheaf $E$ (see \cite[Lemma 6.2]{Bri08}). If $\vbf(E)=(r,c_1,s)$ then $Z_{xH,yH}(E)\not\in\R_{\leq0}$ amounts to $y>\frac{1}{r\sqrt{d}}$ when $x=c_1/r$. As a result, we see that for $0<y\leq \frac{1}{\sqrt{d}}$, $\sigma_{0,y}$ is not a stability condition by considering the spherical object $\Ocal_S[1]$, whereas for $x\notin\mathbb{Q}$, $\sigma_{x,y}$ is a stability condition for all $y>0$.
\end{rem}
We will refer to the set of geometric stability conditions of the form $\sigma_{x,y}$ as \emph{the $xy$-plane}. By the definition of $\mathcal{H}_W$ and $\Gamma$, the wall $W$ intersects the $xy$-plane at a semicircle 
\begin{align*}
    \left(x+\frac{1}{\Gamma}\right)^2+y^2=\frac{1}{\Gamma^2}-n^2.
\end{align*}
It is easy to see that as $0<\Gamma<\frac{1}{n}$ increases, we obtain a collection of nested semicircles decreasing in radii.

Let $\Phi_n$ be the autoequivalence of $\Db(S)$ defined in Section~\ref{pre-bm}. Then $\Phi_n(W)$ is a wall for $\vbf':=\Phi_n(\vbf)=(0,n,-1)$. Similarly, any wall $W'$ for $\vbf'$ gives a wall $\Phi_n^{-1}(W')$ for $\vbf$. We say that \emph{$W'$ corresponds to $\Gamma$} if the rational number $\Gamma$ is associated to $\Phi_n^{-1}(W')$ as a wall for $\vbf$. A wall $W'$ for $\vbf'$ corresponding to $\Gamma$ intersects the $xy$-plane at
\begin{align*}
    \left(x+\frac{1}{2dn}\right)^2+y^2=\left(\frac{1}{2dn}\right)^2+\frac{1}{2d^2\left(\frac{1}{\Gamma}-n\right)n}.
\end{align*}
This yields a collection of nested semicircles whose radii increase with $\Gamma$. 

Finally, the following result justifies our use of Bridgeland stability conditions to understand $\Mov(\hilb{dn^2+1})$.

\begin{prop}\cite[Proposition 4.5]{QS22a} and \cite[Theorem 10.8]{BM14a}
All birational models of $\hilb{dn^2+1}$ arise as moduli spaces of stable objects with Mukai vector $\vbf$ and their birational transformations are induced by crossing flopping walls for $\vbf$ in $\Stab^\dagger(S)$ (or the $xy$-plane).

The above statements remain true if we replace $\hilb{dn^2+1}$ with $M(0,n,-1)$ and $\vbf$ with $\vbf'$.
\end{prop}
The idea of the proof is as follows: fix a small positive irrational number $\epsilon$. For the first statement, one follows the path $\sigma_{-n+\epsilon,t}$ with $t>0$. Then $M_{\sigma_{-n+\epsilon,t}}(\vbf)$ will go through all of the models of $\hilb{dn^2+1}$. Similarly, for the second statement, one follows the path $\sigma_{-\epsilon,t}$ with $t>0$. Then $M_{\sigma_{-\epsilon,t}}(\vbf')$ will go through all of the models of $M(0,n,-1)$.

Thus to understand the birational geometry of $\hilb{dn^2+1}$ or $M(0,n,-1)$, it suffices to understand the crossing of flopping walls for $\vbf$ or $\vbf'$ in the $xy$-plane. \textbf{For the rest of this paper, by a wall we will mean a flopping wall unless it is specified otherwise.}

\section{Primitive Beauville-Mukai systems $M(0,1,-1)$}
In this section we give a full description of the birational geometry of $M(0,1,-1)$. Set $\vbf'=(0,1,-1)$. Then $\vbf(\Ocal_S)=(1,0,1)$ gives a wall for $\vbf'$ with associated $\Gamma_0=\frac{2d}{2d+1}$. The intersection $W_0'$ of this wall with the $xy$-plane is given by
\begin{align*}
    \left(x+\frac{1}{2d}\right)^2+y^2=\left(\frac{1}{2d}\right)^2+\frac{1}{d}.
\end{align*}
By \cite[Lemma 6.1]{Bay18}, $W'_0$ is the largest wall for $\vbf'$ in the $xy$-plane. Hence there is no wall in $\Mov(\hilb{d+1})$ with $\Gamma>\frac{2d}{2d+1}$ and the ray through $\tilde{H}-\frac{2d}{2d+1}B$ is adjacent to the chamber corresponding to $M(0,1,-1)$ in $\Mov(\hilb{d+1})$. 
\begin{thm}\label{BM1}
The movable cone $\Mov(\hilb{d+1})$ has only two chambers, corresponding to $\hilb{d+1}$ and $M(0,1,-1)$. The chambers are separated by the ray through $\tilde{H}-\frac{2d}{2d+1}B$. 
\end{thm}
\begin{rem}\label{BM10.3}
The theorem also follows from \cite[Proposition 10.3]{BM14a}, whose proof relies on a careful analysis of the stability of $\Ical_\xi$ for $\xi\in \hilb{d+1}$ and the nef cone of $\hilb{d+1}$. We provide an alternative proof.
\end{rem}
\begin{proof}
The wall $W_0$ for $\vbf=(1,0,-d)$ corresponding to $\Gamma=\frac{2d}{2d+1}$ intersects the $xy$-plane in the semicircle given by
\begin{align*}
    \left(x+1+\frac{1}{2d}\right)^2+y^2=\left(1+\frac{1}{2d}\right)^2-1.
\end{align*}
Now $W_0$ intersect the line $x=-1$ at $y^2=\frac{1}{d}$. To prove the theorem, it suffices to show that $W_0$ is the largest wall for $\vbf$ in the $xy$-plane. We can achieve this by showing that the line $x=-1$ does not intersect any wall at $y>\frac{1}{\sqrt{d}}$. Suppose otherwise, and that at $(-1,y)$ the bigger wall corresponds to the decomposition of Mukai vectors
\begin{align*}
    (1,0,-d)=(a,b,c)+(1-a,-b,-d-c),
\end{align*}
where $a,b,c\in \Z$. Note that the two Mukai vectors on the right must be those of objects in $\Coh^{-1}(S)$, and this implies $0\leq a+b$ and $a+b\leq 1$, respectively. By Lemma~\ref{isStab}, $\sigma_{-1,y}$ is a stability condition. We have
\begin{align*}
    &Z_{-1,y}(1,0,-d)=(dy^2)+\sqrt{-1}(2dy),\\
    &Z_{-1,y}(a,b,c)=(-2db-c+ad(y^2-1))+\sqrt{-1}(2dy(b+a)),\\
    &Z_{-1,y}(1-a,-b,-d-c)=(2db+d+c+(1-a)d(y^2-1))+\sqrt{-1}(2dy(1-a-b)).
\end{align*}
By our assumption, these three complex numbers have the same argument. In particular, $a+b>0$ and $1-a-b>0$. This contradicts $a+b\in\Z$.
\end{proof}
\begin{rem}\label{DisconnectWall}
We notice that the wall for $\vbf$ corresponding to $\Gamma=\frac{2d}{2d+1}$ is disconnected: for a stability condition of the form $\sigma_{x,y}$ on the wall, let $C_{\sigma_{x,y}}$ be its cone of effective classes (see \cite[Proposition 5.5]{BM14b}). One can easily check that the spherical class $(1,-1,d+1)\in C_{\sigma_{x,y}}$ for $x<-1$ and that $(1,-1,d+1)\notin C_{\sigma_{x,y}}$ for $x>-1$ (note that these two components are separated by $x=-1$ as $\sigma_{-1,y}$ is not a stability condition for $y\leq 1/\sqrt{d}$).
\end{rem}
An immediate consequence of the theorem is that the indeterminacy of the birational map $\Phi_1: \hilb{d+1}\dashrightarrow M(0,1,-1)$ can be resolved by letting $\hilb{d+1}$ cross a single wall.
\begin{cor}\label{Rank1Phi}
Crossing the wall $W_0$ given by
$$\left(x+1+\frac{1}{2d}\right)^2+y^2=\left(1+\frac{1}{2d}\right)^2-1$$
at $x>-1$ induces a stratified flop $F_1: \hilb{d+1}\dashrightarrow \hilb{d+1}_\dagger$. Moreover, the induced birational map $\Phi_1:\hilb{d+1}_\dagger\dashrightarrow M(0,1,-1)$ is actually an isomorphism.
\end{cor}
\begin{proof}
 When $x>-1$, $W_0$ is a flopping wall which is not totally semistable by Remark~\ref{DisconnectWall}. For $\epsilon>0$ small, let $\sigma_t=\sigma_{-1+\epsilon,t}$ for $t>0$. Suppose $\sigma_{t}\cap W_0=\sigma_{t_0}$. Then $M_{\sigma_{t_+}}(\vbf)=\hilb{d+1}$ for any $t_+>t_0$. Let $0<t_-<t_0$ be close to $t_0$ and $\hilb{d+1}_\dagger:=M_{\sigma_{t_-}}(\vbf)$. It is clear that $\hilb{d+1}_\dagger$ does not depend on the choices of $\epsilon$ and $t_-$.

The exceptional locus of $M(0,1,-1)$ for crossing the wall with $\Gamma=\frac{2d}{2d+1}$ was described in \cite[Section 6-8]{Bay18} as the Brill-Noether locus $\mathrm{BN}_{d-1}(|H|)$, with the natural stratification given by the number of global sections. Similarly, one can show that the exceptional locus in $M_{\sigma_{t_+}}$ for crossing the wall $W_0$ is
$$E_1:=\{\xi\in\hilb{d+1}\;|\; h^0(\Ical_\xi(1))\geq2\}.$$
Note that $E_1$ also has a natural stratification, given by $h^0(\Ical_\xi(1))$, i.e., the dimension of the subsystem of curves in $|H|$ containing $\xi$. The wall-crossing map $F_1: \hilb{d+1}\dashrightarrow \hilb{d+1}_\dagger$ is then the stratified flop at $E_1$.\ 

The equivalence $\Phi_1$ induces an isomorphism 
\begin{align*}
    M_{\sigma_{t_-}}(\vbf)\cong M_{\Phi_1(\sigma_{t_-})}(\vbf').
\end{align*}
We note that $\Phi_1(-1,1,-d)=(0,0,1)$, and $\sigma_{t_-}$-stable objects with Mukai vector $(-1,1,-d)$ are of the form $\RHom(\Ical_p,\Ocal_S(-1))[1]$ for $p\in S$. It is easy to check that
$$\Phi_1(\RHom(\Ical_p,\Ocal_S(-1))[1])\cong \Ocal_p,$$
and thus all skyscraper sheaves are $\Phi_1(\sigma_{t_-})$-stable and $\Phi_1(\sigma_{t_-})$ is geometric. Since the action of $\widetilde{\mathrm{GL}}_2^+(\R)$ does not change objects in the moduli space, $M_{\Phi_1(\sigma_{t_-})}(\vbf')=M_{\sigma_{x',y'}}(\vbf')$ for some $x'\in\R$ and $y'>0$. It is clear that $(x',y')$ is above the wall with $\Gamma=\frac{2d}{2d+1}$ for $\vbf'$. Thus $M_{\sigma_{x',y'}}(\vbf')=M(0,1,-1)$. In conclusion, $\Phi_1:\hilb{d+1}_\dagger\to M(0,1,-1)$ is an isomorphism. 
\end{proof}

\section{Rank one walls}\label{Rank1Wall}
Let $n\geq 2$ throughout this section. Set $\vbf=(1,0,-dn^2)$ and $\vbf'=(0,n,-1)$. We introduce so called `rank one' walls, which are always present in the movable cone of $\hilb{dn^2+1}$. These are, roughly, flopping walls given by Mukai vectors of rank one (a precise definition of the rank of a wall will be given later, Definition~\ref{RankDef}). The corresponding flops and their exceptional loci will be easy to describe, see Section~7. We divide the rank one walls into three types: the Brill-Noether (middle) wall, those on the Hilbert scheme side, and those on the Beauville-Mukai system side.

\subsection{The Brill-Noether (middle) wall} We start with the reincarnation of the unique flopping wall of the primitive Beauville-Mukai system in the case when $n\geq 2$, which we call the Brill-Noether wall. Let $\sbf'=\vbf(\Ocal_S)=(1,0,1)$. Then by \cite[Theorem 5.7(b)]{BM14b}, $\sbf'$ gives a flopping wall $\Wcal'$ for $\vbf'$ corresponding to $0<\Gamma_0=\frac{2dn}{2dn^2+1}<\frac{1}{n}$. Using  similar arguments as in \cite[Section 6-8]{Bay18}, one can see that $\Wcal'$ induces a stratified Mukai flop at (some strict transform of) the Brill-Noether locus
$$\Ycal_0=\{\Ecal\in M(0,n,-1)\;|\; h^0(\Ecal)\geq 1 \}.$$
See also the examples in Section 8.

Equivalently, let $\sbf=\Phi_n^{-1}(\sbf')=-\vbf(\Ocal_S(-n))= (-1,n,-dn^2-1)$. Then $\sbf$ gives rise to a flopping wall $\Wcal$ for $\vbf$ corresponding to the same $\Gamma_0=\frac{2dn}{2dn^2+1}$. Similarly to the previous section, $\Wcal$ induces a stratified Mukai flop at (some strict transform of) the locus
$$Y_0=\{\xi\in \hilb{dn^2+1}\;|\; h^0(\Ical_\xi(n))\geq 2 \}.$$

Next we describe two other types of rank one walls located on opposite sides of the Brill-Noether wall in the movable cone.

\subsection{The $\hilb{dn^2+1}$ side.}
Let $\abf=(1,-c_1,s)$, where $1\leq c_1\leq n-1$ and
$$2dnc_1-dn^2+1\leq s\leq dc_1^2+1.$$
Let $\bbf=\vbf-\abf=(0,c_1,-dn^2-s)$. We have
\begin{itemize}
    \item if $s=dc_1^2+1$, then $\abf^2=-2$ and $0<(\abf,\vbf)=dn^2-s<\frac{\vbf^2}{2}$;
    \item if $s\leq dc_1^2$, then $\abf^2=2dc_1^2-2s\geq0$, $\bbf^2=2dc_1^2>0$, and $0<(\abf,\vbf)=dn^2-s<\vbf^2$. In addition, 
    \begin{align*}
        \abf^2\cdot \vbf^2-(\abf,\vbf)^2=(2dnc_1+dn^2+s)(2dnc_1-dn^2-s)<0.
    \end{align*}
\end{itemize}

By \cite[Theorem 5.7(b)]{BM14b}, $\abf$ gives a flopping wall for $\vbf$. One can easily compute that the corresponding $0<\Gamma=\frac{2dc_1}{dn^2+s}<\Gamma_0$. We note that different $(c_1,s)$ can result in the same $\Gamma$, hence the same wall. An example of this situation is the wall with $\Gamma=\frac27$ in \cite{QS22a}.\ 

\begin{rem}\label{rem1}
As we will see in the proof of the next proposition, we will cross the above walls along the path $x=-n$, and $\vbf=\abf+\bbf$ is the induced decomposition of $\vbf$ along $x=-n$. The bound $1\leq c_1\leq n-1$ arises when we require $\abf$ and $\bbf$ to be the Mukai vectors of objects in $\Coh^{-n}(S)$. The inequalities for $s$ are designed to ensure that $\abf$ gives a flopping wall for $\vbf$. Indeed, $2dnc_1-dn^2+1\leq s$ is equivalent to the hyperbolicity of the lattice generated by $\vbf$ and $\abf$, while $s\leq dc_1^2+1$ is equivalent to the Bogomolov inequality $\abf^2\geq -2$. 
\end{rem}

Among these rank one walls, we have 
\begin{align*}
    \Gamma&=\frac{2dc_1}{dn^2+s}\geq \frac{2dc_1}{dn^2+dc_1^2+1}=\frac{2}{c_1+\frac{n^2+\frac1d}{c_1}}.
\end{align*}
 The function $f(x)=x+\frac{n^2+\frac1d}{x}$ is decreasing on the interval $[1,n-1]$, thus $\Gamma\geq\frac{2}{1+n^2+\frac1d}$, where the right hand side can be realized by the wall given by $(1,-1,d+1)$. Thus the wall given by $(1,-1,d+1)$ is the largest among rank one walls. We now show that it is in fact largest among all walls for $\vbf$.
\begin{prop}\label{LargestWallS}
The largest wall in the $xy$-plane for $\vbf$ is the rank one wall given by $(1,-1,d+1)$. 
\end{prop}
\begin{rem}
Similarly to Proposition~\ref{BM1}, this proposition also follows from \cite[Proposition 10.3]{BM14a}. We give here a purely numerical proof, which is more aligned with the similar situation on the Beauville-Mukai system side. Our proof will motivate the definition of the rank of a wall (Definition~\ref{RankDef}), and give a bound for the rank (Lemma~\ref{RankGen}).
\end{rem}
\begin{proof}
For any $\Gamma$, the wall in $\Stab^\dagger(S)$ for $\vbf$ corresponding to $\Gamma$ intersects the $xy$-plane at
\begin{align*}
    \left(x+\frac{1}{\Gamma}\right)^2+y^2=\frac{1}{\Gamma^2}-n^2.
\end{align*}
We only need to consider $0<\Gamma<1/n$. Note that $(1,-1,d+1)$ corresponds to $\Gamma_+=\frac{2d}{dn^2+d+1}$. Suppose that there is a wall with $\Gamma<\gamma$ for some $0<\gamma\leq\Gamma_0=\frac{2dn}{2dn^2+1}$ (later we will let $\gamma=\Gamma_+$). Such a wall intersects $x=-n$ at $y^2=2n\left(\frac{1}{\Gamma}-n\right)$, inducing a decomposition of Mukai vectors
\begin{align*}
    (1,0,-dn^2)=(a,b,c)+(1-a,-b,-dn^2-c),
\end{align*}
where $a,b,c\in \Z$. Note that the two Mukai vectors on the right must be those of objects in $\Coh^{-n}(S)$, which implies $0\leq b+na\leq n$. Furthermore, we have
\begin{align*}
    &Z_{-n,y}(1,0,-dn^2)=dy^2+\sqrt{-1}(2dny),\\
    &Z_{-n,y}(a,b,c)=(-2dbn-c+ad(y^2-n^2))+\sqrt{-1}(2dy(b+an)).
\end{align*}
Since the above two complex numbers must have the same argument, $b+an>0$. A similar computation applied to $(1-a,-b,-dn^2-c)$ shows that $b+an<n$. Defining $i:=b+an$, we therefore have $1\leq i\leq n-1$.

Without loss of generality, we can assume $a\geq 1$ and $\mathrm{gcd}(a,b,c)=1$. The wall condition gives
\begin{align*}
    c&=-2bdn-adn^2-\frac{bdy^2}{n}\\
    &=adn^2-2idn+\frac{(an-i)dy^2}{n}.
\end{align*}
Let $M:=\frac{(an-i)dy^2}{n}\in\Z$. Since $(a,b,c)^2\geq -2$, we get $M\leq \frac{i^2d+1}{a}$.  Since
$$y^2=2n\left(\frac{1}{\Gamma}-n\right)>2n\left(\frac{1}{\gamma}-n\right),$$
we get $M>2d(an-i)\left(\frac{1}{\gamma}-n\right)$. Altogether we obtain
\begin{align}\label{ineqI}
    2d(an-i)\left(\frac{1}{\gamma}-n\right)<M\leq\frac{i^2d+1}{a},
\end{align}
and in particular
\begin{align}\label{ineqII}
    2d(an-i)\left(\frac{1}{\gamma}-n\right)<\frac{i^2d+1}{a}.
\end{align}
Solving this quadratic inequality in $a$, we obtain
\begin{align}\label{ineqcomp}
    a<\frac{1}{2n}\left(i+\sqrt{i^2+\dfrac{2n(i^2d+1)}{d\left(\frac{1}{\gamma}-n\right)}}\right).
\end{align}
Substituting in $i\leq n-1$ and $\gamma=\Gamma_+=\frac{2d}{dn^2+d+1}$, we obtain
\begin{align*}
    a&<\frac12\left(\frac{n-1}{n}+\sqrt{\frac{(n-1)^2}{n^2}+\dfrac{4((n-1)^2d+1)}{n(d(n-1)^2+1)}}\right) \\
    &<\frac{1}{2}\left(1+\sqrt{1+4}\right)<2.
\end{align*}
Thus $a=1$. Using the conditions $\abf^2\geq-2$ and $\abf^2\cdot\vbf^2-(\abf,\vbf)^2<0$, we see that the wall must be one of the rank one walls defined in this subsection. But $\Gamma_+$ corresponds to the largest rank one wall.
\end{proof}

\subsection{The $M(0,n,-1)$ side} Set $\vbf'=(0,n,-1)$. Let $\abf'=(1,c_1,s)$, where $1\leq c_1\leq n-1$ and $1\leq s\leq \min{(dc_1^2+1,d(n-c_1)^2)}$. Let $\bbf'=\vbf'-\abf'=(-1,n-c_1,-1-s)$. Then
\begin{align*}
    \abf'^2=2dc_1^2-2s\geq -2 \qquad\mbox{and}\qquad \bbf'^2=2d(n-c_1)^2-2(s+1)\geq -2.
\end{align*}
Moreover
\begin{itemize}
    \item if $\abf'^2=-2$, then $s=dc_1^2+1$ and $dc_1^2+1\leq d(n-c_1)^2$; we obtain $c_1< \frac{n}{2}$ and $0<(\abf',\vbf')=2dnc_1+1\leq dn^2$;
    \item if $\bbf'^2=-2$, then $s=d(n-c_1)^2$ and $d(n-c_1)^2\leq dc_1^2+1$; we obtain $c_1\geq \frac{n}{2}$ and $0<(\bbf',\vbf')=2dn(n-c_1)-1<dn^2$;
    \item otherwise, $\abf'^2\geq0$, $\bbf'^2\geq 0$, $(\abf',\vbf')>0$, and $(\bbf',\vbf')>0$; moreover
    \begin{align*}
        \abf'^2\vbf'^2-(\abf',\vbf')^2=-4dn(sn+c_1)-1<0.
    \end{align*}
\end{itemize}
By \cite[Theorem 5.7(b)]{BM14b}, $\abf'$ yields a flopping wall for $\vbf'$. One can easily compute that the corresponding $\Gamma_0<\Gamma=\frac{2d(ns+c_1)}{2dn(ns+c_1)+1}<\frac{1}{n}$. We will refer to these walls as rank one walls for $M(0,n,-1)$ or $\vbf'$. 

\begin{rem}
As we will see in the proof of the next proposition, we will cross the above walls along the path $x=0$, and $\vbf'=\abf'+\bbf'$ is the induced decomposition of $\vbf'$ along $x=0$. The bounds on $c_1$ and $s$ arise similarly as in Remark \ref{rem1}
\end{rem}

Our next result shows that the largest wall on the $M(0,n,-1)$ side is a rank one wall. 
\begin{prop}\label{LargestWallM}
The largest wall in the $xy$-plane for $\vbf'$ is a rank one wall given by
\begin{itemize}
    \item $\left(1,\frac{n}{2}, d\left(\frac{n}{2}\right)^2\right)$ if $n$ is even,
    \item $\left(1,\frac{n-1}{2},d\left(\frac{n-1}{2}\right)^2+1\right)$ if $n$ is odd.
\end{itemize}
\end{prop}
\begin{proof}
For any $0<\Gamma<\frac{1}{n}$, the wall in $\Stab^\dagger(S)$ for $\vbf'$ corresponding to $\Gamma$ intersects the $xy$-plane at
\begin{align*}
    \left(x+\frac{1}{2dn}\right)^2+y^2=\left(\frac{1}{2dn}\right)^2+\frac{\Gamma}{2d^2(1-n\Gamma)n}.
\end{align*}
It is clear that larger $\Gamma$ give larger walls and all walls intersect $x=0$.
  A wall at $(0,y)$ for $\vbf'$ induces a decomposition of Mukai vectors
\begin{align*}
    (0,n,-1)=(a,b,c)+(-a,n-b,-1-c),
\end{align*}
 where $a,b,c\in \Z$ and the two Mukai vectors on the right are those of objects in $\Coh^{0}(S)$. As in the previous proof, we obtain $0<b<n$. The wall condition gives
 \begin{align*}
     c=ady^2-\frac{b}{n}.
 \end{align*}
 In particular, $a\neq 0$. 
 
If $|a|=1$, then both $(a,b,c)$ and $(-a,n-b,-1-c)$ are primitive. Using $(a,b,c)^2\geq -2$, $(-a,n-b,-1-c)^2\geq -2$, and the hyperbolicity of the lattice, we see that the wall is either one of the rank one walls given above or it is given by $(1,c_1,0)$, which is smaller than the middle wall (in fact it is a rank one wall on the $\hilb{dn^2+1}$ side). The wall whose lattice contains $(1,c_1,s)$ and $\vbf'$ intersects the vertical line $x=0$ at $y^2=\frac{ns+c_1}{nd}$. It is easy to see that the walls in the proposition are the ones that maximize $ns+c_1$ under the constraints $1\leq c_1\leq n-1$ and $s\leq \min{(dc_1^2+1,d(n-c_1)^2)}$, thus the largest rank one walls.

For $|a|>1$, we assume without loss of generality that $1\leq b\leq \frac{n}{2}$ and that $(a,b,c)$ is primitive. Suppose that $n$ is even; the wall corresponding to $\left(1,\frac{n}{2}, d\left(\frac{n}{2}\right)^2\right)$ intersects $x=0$ at $y_e^2=\frac{n^2}{4}+\frac{1}{2d}$. It suffices to show that $x=0$ does not intersect any such wall at $y>y_e$. Since $(a,b,c)^2\geq -2$, we obtain 
\begin{align*}
    y^2&\leq \frac{b^2}{a^2}+\frac{1}{a^2d}+\frac{b}{nad}\\
    &\leq \frac{n^2}{16}+\frac{1}{4d}+\frac{1}{4d}\\
    &<y^2_e.
\end{align*}
Suppose that $n$ is odd; the wall corresponding to $\left(1,\frac{n-1}{2}, d\left(\frac{n-1}{2}\right)^2+1\right)$ intersects $x=0$ at $y_o^2=\left(\frac{n-1}{2}\right)^2+\frac{3}{2d}-\frac{1}{2dn}$. Note that our assumption $b\leq \frac{n}{2}$ implies $b\leq \frac{n-1}{2}$.  It suffices to show that $x=0$ does not intersect any such wall at $y>y_o$. Since $(a,b,c)^2\geq -2$, we obtain
\begin{align*}
    y^2&\leq \frac{b^2}{a^2}+\frac{1}{a^2d}+\frac{b}{nad}\\
    &\leq \frac{(n-1)^2}{16}+\frac{1}{4d}+\frac{1}{4d}-\frac{1}{4nd}\\
    &<y_o^2.
\end{align*}
This completes the proof.
\end{proof}
Finally we can make explicit the meaning of the rank of a wall.
\begin{defn}\label{RankDef}
A flopping wall for $\vbf$ (or $\vbf'$) is a \emph{rank $r$ wall} if its corresponding $\Gamma$ is less than (respectively, greater than) $\Gamma_0$ and the smallest positive rank of the Mukai vectors used in the decomposition of $\vbf$ (or $\vbf'$) at its intersection with $x=-n $ (respectively, $x=0$) is $r$. We regard the Brill-Noether wall ($\Gamma=\Gamma_0$) as a rank one wall for both $\vbf$ and $\vbf'$.
\end{defn}
It is easy to verify that all rank one walls for $\vbf$ and $\vbf'$ are those listed in this section. Recycling (\ref{ineqcomp}) and noting that $i\leq n-1$ there, we obtain the following rough upper bound for the rank of a wall for $\vbf$.
\begin{lem}\label{RankGen}
Let $0<\gamma\leqslant\Gamma_0$. A wall for $\vbf$ with $\Gamma<(\leqslant) \gamma$ has rank 
\begin{align*}
    r<(\leq)\frac{1}{2n}\left(n-1+\sqrt{(n-1)^2+\dfrac{2n((n-1)^2d+1)}{d\left(\frac{1}{\gamma}-n\right)}}\right).
\end{align*}
\end{lem}

\begin{rem}
Inequality (\ref{ineqI}), along with the fact that $M$ is an integer, often imposes a stronger restriction on the rank than (\ref{ineqcomp}). We will use this observation frequently.
\end{rem}
\begin{rem}
One can obtain a rough bound for the rank of a wall for $\vbf'$ in a similar manner. See Lemma~\ref{RankBoundM}.
\end{rem}

\begin{rem}
The moduli spaces $M(0,2,-1)$ and $M(0,3,-1)$ for $d=1$ are studied in~\cite{Hel20,QS22a}, respectively; all walls turn out to be rank one walls. One may wonder whether all walls for $M(0,n,-1)$ when $d=1$ are rank one walls. The answer is negative: when $n=4$ there is a rank two wall for $M(0,4,-1)$ when $d=1$.
\end{rem}

\section{Totally semistable walls}
In this section a wall does not have to be a flopping wall, i.e., it can be a fake wall, a divisorial wall, or a flopping wall. A wall for $\wbf$ is a \emph{totally semistable wall} if $M^{st}_{\sigma}(\wbf)=\emptyset$ for $\sigma$ on the wall. The behavior of totally semistable walls is hard to control. We will eliminate such walls on the paths where we will carry out wall-crossing. The benefit of having non-totally semistable wall is that the moduli spaces on both sides of the wall will share an open set of objects which are stable across the wall.

By \cite[Theorem 5.7]{BM14b}, a wall $W$ with associated lattice $\Hcal$ is totally semistable if and only if there exists either an isotropic $\ibf\in\Hcal$ with $(\ibf,\wbf)=1$ or an effective spherical $\sbf\in\Hcal$ with $(\sbf,\wbf)<0$. (See~\cite[Proposition~5.5]{BM14b} and the comment immediately after for the definition of effective. We will only need that it implies that the imaginary parts of the central charges of $\sbf$ and $\wbf$ have the same sign.) 
\begin{lem}\label{SnoTSS}
Let $\sigma_{-n,y}$ be a stability condition. Then $\sigma_{-n,y}$ is not on a totally semistable wall for $\vbf$.
\end{lem}
\begin{proof}
Assume $\sigma_{-n,y}$ is on a totally semistable wall $W$ for $\vbf$ with associated lattice $\Hcal$. For the wall to contain $(-n,y)$, $0<\Gamma_W<\frac{1}{n}$. Suppose there is an isotropic $\ibf=(i_0,i_1,i_2)\in\Hcal$ with $(\ibf,\vbf)=1$. Then $di_1^2-i_0i_2=0$ and $dn^2i_0-i_2=1$. Therefore
\begin{align*}
    -i_0=i_0i_2-dn^2i_0^2=di_1^2-dn^2i_0^2=d(i_1+ni_0)(i_1-ni_0).
\end{align*}
Noting that $n\geq 2$ and $d\geq1$, it is easy to see that one can only have $\ibf=(0,0,-1)$. Then $\Gamma_W=0$ and $W$ will not show up in the $xy$-plane.

Suppose that there is an effective spherical $\sbf=(s_0,s_1,s_2)\in\Hcal$ with $(\sbf,\vbf)<0$. Note that the imaginary part of $Z_{-n,y}(\vbf)$ is  $2dny>0$. Then
\begin{align*}
    ds_1^2-s_0s_2=-1&\qquad \mbox{as }\sbf \mbox{ is spherical}, \\
    dn^2s_0-s_2<0&\qquad \mbox{as }(\sbf,\vbf)<0, \\
    s_1+ns_0>0&\qquad \mbox{as }\sbf \mbox{ is effective}.
\end{align*}
By the first equation, $s_0\neq 0$. If $s_0>0$, then $s_2\geq dn^2s_0+1$ and $ns_0\geq-s_1$. Moreover, since $\Gamma_W=\frac{-2ds_1}{dn^2s_0+s_2}>0$, $s_1<0$. Thus
\begin{align*}
    dn^2s_0^2> ds_1^2=s_0s_2-1>dn^2s_0^2+s_0-1,
\end{align*}
which is absurd. If $s_0<0$, then $0<-s_2\leq dn^2(-s_0)-1$ and $s_1> n(-s_0)$. Thus
\begin{align*}
   dn^2s_0^2\leq ds_1^2=s_0s_2-1<dn^2s_0^2+s_0-1,
\end{align*}
which is also absurd.
\end{proof}

We use a slightly more complicated path for the wall-crossing for $\vbf'$.

\begin{lem}\label{MnoTSS}
Fix $y_0>0$. Let $0\leq\epsilon\ll1$ and $y>y_0$. If $\sigma_{-\epsilon,y}$ is a stability condition, then $\sigma_{-\epsilon,y}$ is not on a totally semistable wall for $\vbf'$.
\end{lem}
\begin{rem}
To be precise, we require $\epsilon\geq 0$ to be sufficiently small that $\epsilon\leq\frac{1}{2dn}$ and $d\epsilon^2-\frac{\epsilon}{n}+dy_0^2>0$.
\end{rem}
\begin{proof}
Suppose $\sigma_{-\epsilon,y}$ is on a totally semistable wall $W'$ for $\vbf'$ with associated lattice $\Hcal'$. By applying $\Phi^{-1}_n$ and use the first part of the previous proof, we see that there is no isotropic $\ibf\in\Hcal'$ with $(\ibf,\vbf')=1$. Then there must be an effective spherical $\sbf=(s_0,s_1,s_2)\in\Hcal'$ with $(\sbf,\vbf')<0$. Note that the imaginary part of $Z_{-\epsilon,y}(\vbf')$ is $2dny>0$. Then
\begin{align*}
    ds_1^2-s_0s_2=-1&\qquad \mbox{as }\sbf \mbox{ is spherical}, \\
    2dns_1+s_0<0&\qquad \mbox{as }(\sbf,\vbf')<0,\\
    s_1+\epsilon s_0>0&\qquad \mbox{as }\sbf \mbox{ is effective}.
\end{align*}
Combining $s_1+\epsilon s_0> 0$, $2dns_1+s_0<0$, and $\epsilon\geq 0$ gives
$$-s_1< \epsilon s_0\leq\epsilon (-2dns_1)=2dn\epsilon (-s_1).$$
If $s_1<0$ we get $1< 2dn\epsilon$, which is impossible if $\epsilon$ is small enough. It is clear that $s_1$ cannot be $0$. Thus $s_1>0$ and $s_0<-2dns_1<0$. By the wall condition we have
\begin{align*}
    s_2=s_0\left(d\epsilon^2-\frac{\epsilon}{n}+dy^2\right)-\frac{s_1}{n}.
\end{align*}
When $\epsilon$ is small relative to $y_0$, $s_2<-\frac{s_1}{n}$. Thus
$$ds_1^2=(-s_0)(-s_2)-1>(2dns_1)\left(\frac{s_1}{n}\right)-1= 2ds_1^2-1,$$
which is absurd.
\end{proof}
\begin{rem}
A bonus of the previous two lemmas is that we also eliminate fake walls along our paths of wall-crossings since fake walls must be totally semistable by \cite[Theorem 5.7]{BM14b}.
\end{rem}

\section{Rank two Beauville-Mukai systems $M(0,2,-1)$}
As in the preliminaries section, the autoequivalence $\Phi_2$ of $\Db(S)$ induces a rational map $\Phi_2: \hilb{4d+1}\dashrightarrow M(0,2,-1)$. In this section we try to understand the birational geometry of $M(0,2,-1)$ by studying both the $\hilb{4d+1}$ side and the $M(0,2,-1)$ side. We will use the rank one walls studied above as anchors for our analysis. Fix $\vbf=(1,0,-4d)$ and $\vbf'=(0,2,-1)$. The middle wall corresponds to  $\Gamma_0=\frac{4d}{8d+1}$. 
\subsection{The $\hilb{4d+1}$ side}
By the previous section, we have $d+1$ rank one walls for $\vbf$ given by the Mukai vectors $(1,-1,k)$ for $1\leq k\leq d+1$. By Proposition~\ref{LargestWallS}, the wall corresponding to $(1,-1,d+1)$ is the largest wall for $\vbf$ in the $xy$-plane. Let $0<\gamma\leqslant\Gamma_0$. A wall for $\vbf$ with $\Gamma<(\leqslant) \gamma$ has rank
\begin{align}\label{RankBoundSn2}
    a<(\leq)\frac14\Bigg(1+\sqrt{1+\dfrac{4(d+1)}{d\left(\frac{1}{\gamma}-2\right)}}\Bigg)
\end{align} 
by Lemma~\ref{RankGen}. Moreover, the inequality (\ref{ineqI}) becomes, 
\begin{align}\label{ineq1}
    2d(2a-1)\left(\frac{1}{\gamma}-2\right)<M\leq\frac{d+1}{a}
\end{align}
for some integer $M$, since $n=2$ and $i=1$. The inequality (\ref{ineqII}) becomes
\begin{align}\label{ineq2}
    2d(2a-1)\left(\frac{1}{\gamma}-2\right)<\frac{d+1}{a}.
\end{align}

Note that the wall $W_-$ given by $(1,-1,1)$ is the smallest among rank one walls, with corresponding $\Gamma=\frac{2d}{4d+1}$ largest among rank one walls. An immediate consequence of the above inequalities is that when the degree of the K3 surface is small, there are no other walls in between these rank one walls.
\begin{cor}\label{RankOne}
Suppose $d\leq 6$. Any wall for $\vbf$ with $\Gamma\leq\frac{2d}{4d+1}$ is a rank one wall.
\end{cor}
\begin{proof}
Since we know that the wall with $\Gamma=\frac{2d}{4d+1}$ is rank one, we can assume $\Gamma<\frac{2d}{4d+1}$. Let $a$ be the rank of such a wall. Set $\gamma=\frac{2d}{4d+1}$.  If $a\geq 2$,  then (\ref{ineq2}) leads to $d+1>(2a-1)a\geq 6$. Thus $d$ must be at least $6$ for $W$ to be not rank one. Moreover, when $d=6$, there is no integer $M$ satisfying
$$2d(2a-1)\left(\frac{1}{\gamma}-2\right)=2a-1<M\leq\frac{d+1}{a}$$
for any $a\geq 2$. Thus for $d\leq 6$, all walls with $\Gamma\leq \frac{2d}{4d+1}$ must be rank one walls.
\end{proof}
\begin{exmp}
Let $d=7$. We have a wall $W_\sharp$ inducing the following decomposition of Mukai vectors at its intersection with $x=-2$:
\begin{align*}
    (1,0,-28)=(2,-3,32)+(-1,3,-60).
\end{align*}
This wall has corresponding $\Gamma=\frac{21}{44}$. Recall that the rank one wall given by $(1,-1,k)$ has $\Gamma=\frac{2d}{4d+k}$. Since 
\begin{align*}
    \frac{2\cdot7}{4\cdot7+2}<\frac{21}{44}<\frac{2\cdot7}{4\cdot7+1},
\end{align*}
$W_\sharp$ is not a rank one wall. 
\end{exmp}
Recall that the Brill-Noether wall is given by the spherical class $\mathbf{s}=(-1,2,-4d-1)$. Note that $(\sbf,\vbf)=1$. This wall has corresponding $\Gamma=\frac{4d}{8d+1}$ and intersection $\Wcal$ with $xy$-plane given by
\begin{align*}
    \left(x+2+\frac{1}{4d}\right)^2+y^2=\left(2+\frac{1}{4d}\right)^2-4.
\end{align*}
\begin{cor}\label{OnleOneTwoS}
Suppose $d\leq 7$. Any wall for $\vbf$ with $\Gamma< \frac{4d}{8d+1}$ has rank one or two.
\end{cor}
\begin{proof}
Let $a$ be the rank of such a wall and  set $\gamma=\frac{4d}{8d+1}$. By (\ref{ineq1}), $\frac{2a-1}{2}<M\leq\frac{d+1}{a}$. Thus if $a\geq 3$, then $d>\frac{(2a-1)a}{2}-1\geq \frac{13}{2}$. Moreover, if $d=7$ and $a\geq 3$, then there is no integer $M$ so that $\frac{2a-1}{2}<M\leq \frac{d+1}{a}$.
\end{proof}
Lastly, we characterize rank two walls between $W_-$ and $\Wcal$.
\begin{lem}\label{Rank2S}
The only (potential) rank two wall for $\vbf$ between $W_-$ and $\Wcal$ is given by the Mukai vector $(2,-3,4d+2)$. Such a wall exists if and only if $d\geq 3$.
\end{lem}
\begin{proof}
We adapt the notation from the proof of Proposition~\ref{LargestWallS}. Suppose that the vertical line $x=-2$ intersects a rank two wall between $W_-$ and $\Wcal$ at $(-2,y)$. Then $\frac{1}{d}<y^2<\frac{2}{d}$. Now
\begin{align*}
    4d+\frac{3}{2}<c=4d+\frac{3}{2}y^2d<4d+3.
\end{align*}
Since $c\in \Z$, we must have $c=4d+2$. It follows that the wall is given by $\abf=(2,-3,4d+2)$. 

Since $0<(\abf,\vbf)=4d-2<\frac{\vbf^2}{2}$, such a wall exists if and only if the following conditions are satisfied:
\begin{itemize}
    \item $\abf^2=2d-8\geq-2$,
    \item $(\vbf-\abf)^2=2d-4\geq -2$,
    \item $\abf^2\cdot\vbf^2-(\abf,\vbf)^2=-4-48d<0$.
\end{itemize}
It is clear that these amount to $d\geq 3$.
\end{proof}

To conclude this subsection, we provide a quick algorithm for finding all the walls for $\vbf$ with $\Gamma<\Gamma_0=\frac{4d}{8d+1}$. Recall that a wall is associated to a primitive hyperbolic rank two lattice containing $\vbf$. Our algorithm finds an element in the lattice that is not a multiple of $\vbf$.
\begin{prop}\label{AlgS}
Let $\abf=(a,b,c)$ be a primitive Mukai vector with $a\geq 1$. Then there exists a wall for $\vbf$ whose lattice contains $\abf$ so that its corresponding $\Gamma$ satisfies $0< \Gamma<\frac{4d}{8d+1}$ if the following conditions are satisfied:
\begin{enumerate}
    \item $a<\frac14+\sqrt{\frac{1}{16}+d+1}$,
    \item $b=1-2a$,
    \item the integer $M:=c-4d(a-1)$ satisfies $\frac{2a-1}{2}<M\leq \frac{d+1}{a}$.
\end{enumerate}
\end{prop}
\begin{rem}
The three conditions come from the proof of Proposition~\ref{LargestWallS} with $\gamma=\frac{4d}{8d+1}$. The first gives finitely many choices for $a$, the second determines $b$ in terms of $a$, and the third gives finitely many choices for $c$, again in terms of $a$ (of course the inequality in (1) follows from the one in (3)). Any $\abf$ that realizes the rank of such a wall will satisfy these conditions. This guarantees that our algorithm will not miss any walls with $0< \Gamma<\frac{4d}{8d+1}$
\end{rem}
\begin{proof}
By conditions (2) and (3) we have
\begin{align*}
    \abf^2&=2(d-aM)\geq -2, \\
    (\vbf-\abf)^2&=2(d+(1-a)M)>-2, \\
    \abf^2\cdot\vbf^2-(\abf,\vbf)^2&=8dM(1-2a)-M^2<0, \\
    0<(\abf,\vbf)&=4d-M<\frac{\vbf^2}{2}.
\end{align*}
By \cite[Theorem 5.7(b)]{BM14b}, there is a flopping wall for $\vbf$ whose lattice contains $\abf$. The bound for $\Gamma$ follows from the condition $\frac{2a-1}{2}<M$ and the proof of Proposition~\ref{LargestWallS}.
\end{proof}

\subsection{The $M(0,2,-1)$ side} Similarly to the previous subsection, we start with a (rough) upper bound for the rank of a wall for $\vbf'=(0,2,-1)$.
\begin{prop}\label{RankBoundM}
Let $\Gamma_0\leq\gamma'<\frac{1}{2}$. A wall for $\vbf'$ with $\Gamma>(\geqslant) \gamma'$ has rank 
\begin{align*}
    r<(\leqslant)-d\left(\frac{1}{\gamma'}-2\right)+\sqrt{d^2\left(\frac{1}{\gamma'}-2\right)^2+4d(d+1)\left(\frac{1}{\gamma'}-2\right)}.
\end{align*}
\end{prop}
\begin{proof}
We prove the strict inequality version; the non-strict version follows similarly. Throughout this proof, by a wall we mean a wall for $\vbf'$. A wall $W'_{\Gamma}$ corresponding to $\Gamma$ intersects the $xy$-plane in the semicircle given by
\begin{align*}
    \left(x+\frac{1}{4d}\right)^2+y^2=\left(\frac{1}{4d}\right)^2+\frac{\Gamma}{4d^2(1-2\Gamma)}.
\end{align*}
Note that $W'_\Gamma$ intersects the vertical line $x=0$ at
$$y^2=\frac{1}{4d^2\left(\frac{1}{\Gamma}-2\right)}>\frac{1}{4d^2\left(\frac{1}{\gamma'}-2\right)}.$$
By the proof of Proposition~\ref{LargestWallM}, we can suppose that the rank of $W'_\Gamma$ is realized at $(0,y)$ by a decomposition of Mukai vectors
\begin{align*}
    (0,2,-1)=(a,1,c)+(-a,1,-1-c),
\end{align*}
where $a$, $c\in \Z$, $a\geq 1$ (and hence is the rank), and the two Mukai vectors on the right are those of objects in $\Coh^{0}(S)$. The wall condition gives
$$c=ady^2-\half>\frac{a}{4d\left(\frac{1}{\gamma'}-2\right)}-\half.$$
Since $(-a,1,-1-c)^2\geq-2$, we have $c\leq \frac{d+1}{a}-1$. Altogether we obtain
\begin{align}\label{ineq3}
    \frac{a}{4d\left(\frac{1}{\gamma'}-2\right)}-\half<c\leq \frac{d+1}{a}-1,
\end{align}
which leads to 
\begin{align}\label{ineq4}
    \frac{a}{4d\left(\frac{1}{\gamma'}-2\right)}-\half< \frac{d+1}{a}-1,
\end{align}
The proposition follows from solving this quadratic inequality in $a$.
\end{proof}
By the previous section, we have $d$ rank one walls for $\vbf'$ given by $(1,1,k)$ for $1\leq k\leq d$. By Proposition~\ref{LargestWallM}, the wall corresponding to $(1,1,d)$ is the largest wall for $\vbf'$ in the $xy$-plane. Note that the wall $W'_-$ given by $(1,1,1)$ is the smallest among rank one walls, with corresponding $\Gamma=\frac{6d}{12d+1}$. Our next result shows that when the degree of the K3 surface is small, there are no other walls in between the rank one walls.

\begin{cor}\label{RankOneM}
Suppose $d\leq 6$. Any wall for $\vbf'$ with $\Gamma\geq\frac{6d}{12d+1}$ is a rank one wall.
\end{cor}
\begin{proof}
Since we know that the wall $W_-'$ for $\vbf'$ with $\Gamma=\frac{6d}{12d+1}$ has rank one, we can assume $\Gamma>\frac{6d}{12d+1}$. Let $a$ be the rank of such a wall. Set $\gamma'=\frac{6d}{12d+1}$. By the previous proposition,
\begin{align*}
    a<-\frac16+\sqrt{\frac{1}{36}+\frac23(d+1)}\leq 2.
\end{align*}
It follows that $a=1$.
\end{proof}
\begin{exmp}
Let $d=7$. We have a wall $W_\sharp'$ whose intersection with $x=0$ induces the decomposition of Mukai vectors
\begin{align*}
    (0,2,-1)=(2,1,3)+(-2,1,-4).
\end{align*}
This wall has corresponding $\Gamma=\frac{49}{99}$. Recall that the rank one wall for $\vbf'$ whose lattice contains $(1,1,k)$ has $\Gamma=\frac{2d(2k+1)}{4d(2k+1)+1}$. Since 
\begin{align*}
    \frac{3\cdot14}{3\cdot28+1}<\frac{49}{99}<\frac{5\cdot14}{5\cdot28+1},
\end{align*}
$W'_\sharp$ is not a rank one wall for $\vbf'$ and $W'_\sharp$ is in between the rank one walls given by $(1,1,1)$ and $(1,1,2)$. 
\end{exmp}
We now consider the Brill-Noether wall for $\vbf'$ given by $\sbf'=(1,0,1)$, corresponding to $\Gamma=\frac{4d}{8d+1}$. Its intersection $\Wcal'$ with the $xy$-plane is given by
\begin{align*}
    \left(x+\frac{1}{4d}\right)^2+y^2=\frac{1}{16d^2}+\frac{1}{d}.
\end{align*}
\begin{cor}\label{OnlyOneTwoM}
Suppose $d\leq 9$. Any wall for $\vbf'$ with $\Gamma> \frac{4d}{8d+1}$ has rank one or two.
\end{cor}
\begin{proof}
Let $r$ be the rank of such a wall and set $\gamma'=\frac{4d}{8d+1}$. By the proposition,
\begin{align*}
    r<-\frac14+\sqrt{\frac{1}{16}+d+1}\leq -\frac14+\sqrt{\frac{1}{16}+9+1}\approx 2.92.
\end{align*}
\end{proof}
\begin{lem}\label{Rank2M}
The only (potential) rank two wall for $\vbf'$ between $W'_-$ and $\Wcal'$ is given by the Mukai vector $(2,1,2)$. Such a wall exists if and only if $d\geq 5$.
\end{lem}
\begin{proof}
We adapt the notation in the the proof of Proposition~\ref{LargestWallM}. Suppose $x=0$ intersects a rank two wall between $W_-'$ and $\Wcal'$ at $(0,y)$. Then $\frac{1}{d}<y^2<\frac{3}{2d}$. Now
\begin{align*}
    2-\half<c=ady^2-\half<2\cdot\frac{3}{2}-\half.
\end{align*}
Thus $c=2$ and the wall is given by $\abf'=(2,1,2)$.

Noting that $0<(\vbf'-\abf',\vbf')=4d-2<\frac{\vbf'^2}{2}$, such a wall exists if and only if the following conditions are satisfied:
\begin{itemize}
    \item $\abf'^2=2d-8\geq-2$,
    \item $(\vbf'-\abf')^2=2d-12\geq -2$,
    \item $\abf'^2\cdot\vbf'^2-(\abf',\vbf')^2=-80d-4<0$.
\end{itemize}
It is clear that the above amount to $d\geq 5$.
\end{proof}

To conclude this subsection, we provide a fast algorithm for finding all the walls for $\vbf'$ with $\frac{4d}{8d+1}<\Gamma< \frac{1}{2}$. Our algorithm finds an element in the lattice associated to the wall that is not a multiple of $\vbf'$.
\begin{prop}\label{AlgM}
Let $\abf'=(a,b,c)$ be a primitive Mukai vector with $a\geq 1$. Then there exists a wall for $\vbf'$ whose lattice contains $\abf'$ so that its corresponding $\Gamma$ satisfies $\frac{4d}{8d+1}< \Gamma<\half$ if the following conditions are satisfied:
\begin{enumerate}
    \item $a<-\frac14+\sqrt{\frac{1}{16}+d+1}$,
    \item $b=1$,
    \item the integer $c$ satisfies $a-\half<c\leq \frac{d+1}{a}-1$.
\end{enumerate}
\end{prop}
\begin{rem}
The three conditions come from the proof of Proposition~\ref{RankBoundM} with $\gamma'=\frac{4d}{8d+1}$. Any $\abf'$ that realizes the rank of such a wall will satisfy these conditions. This guarantees that our algorithm will not miss any wall with $\frac{4d}{8d+1}< \Gamma<\half$.
\end{rem}
\begin{proof}
By conditions (2) and (3) we have
\begin{align*}
    \abf'^2&=2(d-ac)\geq -2+2a> -2, \\
    (\vbf'-\abf')^2&=2(d-a(1+c))\geq -2, \\
    \abf'^2\cdot\vbf'^2-(\abf',\vbf')^2&=-16acd-8ad-a^2<0, \\
    0<(\vbf'-\abf',\vbf')&=4d-a<\frac{\vbf'^2}{2}.
\end{align*}
By \cite[Theorem 5.7(b)]{BM14b}, there is a flopping wall for $\vbf'$ whose lattice contains $\abf'$. The bound for $\Gamma$ follows from the condition $a-\half<c$ and the proof of Proposition~\ref{RankBoundM}.
\end{proof}
\begin{rem}
Combining Propositions~\ref{AlgS} and \ref{AlgM}, one can find all walls in $\Mov(\hilb{4d+1})$. Moreover, it is easy to see that the number of walls in $\Mov(\hilb{4d+1})$ is $O(d^\frac{3}{2})$.
\end{rem}

\section{Exceptional loci of rank one walls}\label{SectExc}
In this section we study the exceptional loci of the rank one walls in the case when $n=2$. 
\subsection{The $\hilb{4d+1}$ side}
For a subscheme $Z\subset S$, we use $\supp(Z)$ to denote the set-theoretic support of $Z$. For a coherent sheaf $F$ on $S$, we use $\Supp(F)$ to denote the Fitting support of $F$.
\begin{defn}
    Let $\xi\subset S$ be a zero-dimensional subscheme. A closed subscheme $\xi'\subset \xi$ is \it{saturated} if $\Supp(\Ocal_\xi/\Ocal_{\xi'})\cap \supp(\xi')=\emptyset$.
\end{defn}
The following notations and results related to saturated subschemes will also be needed.
\begin{lemdef}
Let $\xi$ be a zero-dimensional closed subscheme of $S$ and let $\xi'\subset \xi$ be a saturated subscheme. Then
\begin{enumerate}
    \item $\Ocal_\xi\cong\Ocal_{\xi'}\oplus(\Ocal_\xi/\Ocal_{\xi'})$. We will denote the subscheme of $\xi$ whose structure sheaf is the second summand by $\xi\backslash\xi'$.
    \item $\xi\backslash\xi'\subset \xi$ is a saturated subscheme.
    \item Suppose $\xi''\subset \xi$ is also saturated. We define $\xi'\cup\xi''$ to be the saturated subscheme of $\xi$ whose support is $\supp(\xi')\cup\supp(\xi'')$.
    \item For any subscheme $\zeta\subset \xi$, we define $\zeta\backslash\xi'$ to be the maximal subscheme of $\zeta$ whose support does not intersect $\supp(\xi')$.
\end{enumerate}
\end{lemdef}

\noindent  For $1\leq k\leq d+1$, we define
\begin{align*}
    Y_k=\{\xi\in\hilb{4d+1}\;|\; \exists \mbox{ saturated }\zeta\subset\xi \mbox{ of length }3d+k \mbox{ and } C\in|H| \mbox{ such that } \zeta\subset C\}.
\end{align*}
Then $Y_k$ is a locally closed subset of $\hilb{4d+1}$ with closure
\begin{align*}
    \overline{Y_k}=\{\xi\in\hilb{4d+1}\;|\; \exists\mbox{ }\zeta\subset\xi \mbox{ of length }3d+k \mbox{ and } C\in|H| \mbox{ such that } \zeta\subset C\}.
\end{align*}
Note that $Y_{d+1}=\overline{Y_{d+1}}$ and $\overline{Y_{k+1}}\subset\overline{Y_{k}}$ for $1\leq k\leq d$.
\begin{prop}\label{WisP^nbundle}
With $Y_k$ defined as above, we have that
\begin{enumerate}
    \item $Y_{d+1}$ is a $\Pbb^{3d}$-bundle over $M(0,1,-5d-1)$,
    \item  $Y_k\backslash\overline{Y_{k+1}}$ is a $\Pbb^{2d+k-1}$-bundle over an open subset of $\hilb{d+1-k}\times M(0,1,-4d-k)$ for $1\leq k\leq d$.
\end{enumerate}
\end{prop}
\begin{proof}
(1) By definition, for any $\xi\in Y_{d+1}$ there exists a curve in $|H|$ containing $\xi$. We claim that such a curve is unique. Suppose $C_1$, $C_2\in |H|$, $C_1\neq C_2$, and $\xi\in C_1\cap C_2$. Since $C_1$ and $C_2$ are both integral, $C_1\cap C_2$ is a zero-dimensional subscheme with length $H^2=2d<4d+1$, hence unable to contain $\xi$. Thus there exists a unique (up to scalar) injection $\Ocal_S(-1)\to\Ical_\xi$, whose cokernel $\Ecal_\xi$ has Mukai vector $(0,1,-5d-1)$. Then $\Ecal_\xi$ is pure of dimension one by the Bogomolov inequality. Since all curves in $|H|$ are integral, $\Ecal_\xi$ is stable. As a result, we have a morphism
\begin{eqnarray*}
\psi_{d+1}:Y_{d+1} & \to & M(0,1,-5d-1), \\
\xi & \mapsto & \Ecal_\xi.
\end{eqnarray*}
It remains to show that $\psi_{d+1}$ is surjective and that its fiber over each point in $M(0,1,-5d-1)$ is isomorphic to $\Pbb^{3d}$. Let $\Ecal\in M(0,1,-5d-1)$; then $\Ext^1(\Ecal,\Ocal_S(-1))=\Cbb^{3d+1}$. Moreover, any non-split extension 
\begin{align*}
    0\to\Ocal_S(-1)\to I\to \Ecal\to 0
\end{align*}
yields a torsion-free sheaf $I$ with Mukai vector $(1,0,-4d)$. Hence $I\cong\Ical_{\xi}$ for some $\xi\in Y_{d+1}$. Thus $\psi_{d+1}$ is surjective and $\psi_{d+1}^{-1}(\Ecal)\cong\Pbb(\Ext^1(\Ecal,\Ocal_S(-1)))\cong\Pbb^{3d}$ for $\Ecal\in M(0,1,-5d-1)$.\\

(2) Let $U_k\subset \hilb{d+1-k}\times M(0,1,-4d-k)$ be the open subset parametrizing pairs $(\eta,\Ecal)$ satisfying $\eta\cap \Supp(\Ecal)=\emptyset$. We will show that $Y_k\backslash\overline{Y_{k+1}}$ is a $\Pbb^{2d+k-1}$-bundle over $U_k$. For $\xi\in Y_k\backslash\overline{Y_{k+1}}$, there exists a saturated subscheme $\zeta\subset\xi$ with length $3d+k$ so that $\zeta\subset C$ for some $C\in|H|$. We claim that both $\zeta$ and $C$ are unique. Suppose that another saturated $\zeta'\subset \xi$ has length $3d+k$ and $\zeta'\subset C'$ for some $C'\in |H|$. Then
$$l(\zeta\cap\zeta')\geq 2(3d+k)-(4d+1)=2d+2k-1\geq 2d+1.$$
If $C\neq C'$, then $C\cap C'$ is a zero-dimensional subscheme with length $2d$, which is impossible since $\zeta\cap\zeta'\subset C\cap C'$. Thus $C=C'$. Now $\zeta\backslash\zeta'\subset\zeta\subset C$. If $\zeta\backslash\zeta'\neq\emptyset$, then $\zeta'\cup (\zeta\backslash\zeta')\subset C$, contradicting our assumption that $\xi\notin\overline{Y_{k+1}}$. Hence $\zeta\backslash\zeta'=\emptyset$, and therefore $\zeta=\zeta'$ as they have the same length. This proves the claim. We obtain a well-defined morphism
\begin{eqnarray*}
    \psi_k: Y_k\backslash\overline{Y_{k+1}} & \to & \hilb{d+1-k}\times M(0,1,-4d-k), \\
    \xi & \mapsto & (\xi\backslash\zeta,\Ical_{\zeta/C}).
\end{eqnarray*}
It is clear that the image of $\psi_k$ is contained in $U_k$. We now show that $\psi_k$ maps onto $U_k$ and the fiber over any point in $U_k$ is isomorphic to $\Pbb^{2d+k-1}$. For $(\eta,\Ecal)\in U_k$, it is easy to check that $\Ext^1(\Ecal,\Ical_\eta(-1))=\C^{2d+k}$. For any non-split extension
\begin{align*}
    0\to \Ical_\eta(-1)\to I\to \Ecal\to 0,
\end{align*}
$I$ must be torsion-free since $\eta\cap \Supp(\Ecal)=\emptyset$. Since $v(I)=(1,0,-4d)$, $I$ is isomorphic to the ideal sheaf of a length $4d+1$ subscheme $\xi$. Moreover, $\eta$ is a saturated subscheme of $\xi$ and $\xi\backslash\eta\subset \Supp(\Ecal)$; thus $\xi\in Y_k$. Suppose $\xi\in \overline{Y_{k+1}}$. Then there exists a (not necessarily saturated) subscheme $\delta\subset\xi$ of length $3d+k+1$ so that $\delta$ is contained in some curve in $|H|$. Now $\delta\backslash\eta$ is a closed subscheme of $\xi\backslash\eta\subset \Supp(\Ecal)$ and
$$l(\delta\backslash\eta)\geq (3d+k+1)-(d+1-k)=2d+2k.$$
As before, any subscheme of length at least $2d+1$ is contained in at most one curve in $|H|$. Hence $\delta\subset\Supp(\Ecal)$. Since $l(\xi\backslash\eta)=3d+k<3d+k+1=l(\delta)$, we must have $\eta\cap\delta\neq \emptyset$, which contradicts our assumption that $\eta\cap\Supp(\Ecal)=\emptyset$. As a result, $\xi\in Y_k\backslash\overline{Y_{k+1}}$. It follows that $\psi_k$ maps onto $U_k$ and the fiber of $\psi_k$ over $(\eta,\Ecal)\in U_k$ is $\Pbb(\Ext^1(\Ecal,\Ical_\eta(-1))\cong\Pbb^{2d+k-1}$.
\end{proof}

Next we show that the proper transforms of the $Y_k$'s are the exceptional loci of the rank one walls for $\vbf$. We will need the following lemma.
\begin{lem}\label{StayGieseker}
The path of stability conditions $\sigma_t:=\sigma_{-2,t}$, where $t>\frac{1}{\sqrt{d}}$, does not cross any wall for a Mukai vector of the form $(A,1-2A,B)$, where $A, B\in \Z$.
\end{lem}
\begin{proof}
Suppose $\sigma_t$ intersects a wall for $(A,1-2A,B)$ at $t=t^\dagger>\frac{1}{\sqrt{d}}$. Such a wall induces a decomposition of Mukai vectors
\begin{align*}
    (A,1-2A,B)=(a,b,c)+(A-a,1-2A-b,B-c),
\end{align*}
where $a,b,c\in \Z$ and the two Mukai vectors on the right are those of objects in $\Coh^{-2}(S)$. As in the proof of Proposition~\ref{LargestWallS}, we get $0<b+2a<1$, which leads to a contradiction.
\end{proof}

\begin{prop}\label{CrossRankOneS}
Suppose $d\leq 6$. One obtains $d+1$ birational models of $\hilb{4d+1}$ by crossing the rank one walls for $\vbf$. These models are connected by a chain of flops
\[ \xymatrix@R-1pc@C-2pc{
& \Bl_{Y_{d+1}}\hilb{4d+1} \ar[dr]\ar[dl] && \Bl_{\tilde{Y}_d}X_{d+1}\ar[dr]\ar[dl] && \Bl_{\tilde{Y}_{k}}X_k\ar[dr]\ar[dl] &&  \Bl_{\tilde{Y_1}}X_2\ar[dr]\ar[dl]\\
\hilb{4d+1} \ar@{-->}[rr]^{f_{d+1}} && X_{d+1} \ar@{-->}[rr]^{f_d} && X_d \cdots\cdot X_{k+1}\ar@{-->}[rr]^{f_{k}} && X_{k}\cdots\cdot X_2\ar@{-->}[rr]^{f_1} &&X_1,} \]
where $\widetilde{\bullet}$ denotes the strict transform of the set $\bullet$ under the appropriate birational map.
\end{prop}
\begin{proof}
We consider wall-crossing along the path of stability conditions $\sigma_t:=\sigma_{-2,t}$, where $\sqrt{\frac{1}{d}}<t<+\infty$. By Lemma~\ref{SnoTSS}, there are no totally semistable walls for $\vbf$ on this path.  By Proposition~\ref{LargestWallS}, $M_{\sigma_t}(\vbf)=\hilb{4d+1}$ for $t>\sqrt{\frac{2d+2}{d}}$. By Corollary~\ref{RankOne}, $\sigma_t$ will intersect only rank one walls, given by $(1,-1,k)$ with $1\leq k\leq d+1$ in decreasing order in $k$, as $t$ goes from $+\infty$ to $\sqrt{\frac{2}{d}}-\delta$ for any small $\delta>0$. We will use $f_k$ to denote the flop induced by crossing the wall given by $(1,-1,k)$ and $X_k$ to denote the birational model obtained from $f_k$ which corresponds to the smaller value of $t$. 

The first wall $\sigma_t$ crosses as $t$ decreases is at $t_{d+1}=\sqrt{\frac{2d+2}{d}}$, inducing the (unique) decomposition of Mukai vectors
\begin{align*}
    (1,0,-4d)=(1,-1,d+1)+(0,1,-5d-1).
\end{align*}
By Lemma~\ref{StayGieseker}, near $t=t_{d+1}$, we have $M_{\sigma_t}(0,1,-5d-1)= M_H(0,1,-5d-1)$ and $M_{\sigma_t}(1,-1,d+1)$ is a single point corresponding to $\Ocal_S(-1)$. Hence an ideal sheaf $\Ical_\xi$ is in the exceptional locus of $f_{d+1}$ in $\hilb{4d+1}$ if and only if it fits into a (non-split) short exact sequence
\begin{align*}
    0\to\Ocal_S(-1)\to\Ical_\xi\to\Ecal\to 0,
\end{align*}
where $\Ecal\in M_H(0,1,-5d-1)$. This is equivalent to $\xi\in Y_{d+1}$. Thus $f_{d+1}$ is the flop of $\hilb{4d+1}$ at $Y_{d+1}$.

For $1\leq k\leq d$, $\sigma_t$ intersections the wall given by $(1,-1,k)$ at $t_k=\sqrt{\frac{2k}{d}}$. The wall induces the (unique) decomposition of Mukai vectors
\begin{align*}
    (1,0,-4d)=(1,-1,k)+(0,1,-4d-k).
\end{align*}
By Lemma~\ref{StayGieseker}, near $t=t_k$ we have $M_{\sigma_k}(1,-1,k)=M_H(1,-1,k)\cong\hilb{d+1-k}$ and we have $M_{\sigma_t}(0,1,-4d-k)=M_H(0,1,-4d-k)$. Hence an ideal sheaf $\Ical_\xi$ is in the exceptional locus of $f_k$ if and only if $\xi\notin\overline{Y_{k+1}}$ (so that $\Ical_\xi\in \hilb{4d+1}\cap X_{k+1})$ and it fits into a (non-split) short exact sequence
\begin{align*}
    0\to\Ical_{\zeta_k}(-1)\to\Ical_\xi\to\Ecal_k\to 0
\end{align*}
where $\zeta_k\in \hilb{d+1-k}$ and $\Ecal_k\in M_H(0,1,-4d-k)$. This is equivalent to $\xi\in\overline{Y_k}\backslash\overline{Y_{k+1}}$. Since the exceptional locus $E_k$ of $f_k$ in $X_{k+1}$ is irreducible (by the decomposition of Mukai vectors) and $\hilb{4d+1}\cap X_{k+1}$ is open in $X_{k+1}$, $E_k$ is the strict transform of $Y_k$. 
\end{proof}

\subsection{The $M(0,2,-1)$ side}
Let $B=|2H|$ and let $B^\circ\subset B$ be the locus of smooth curves. Let $M(0,2,-1)^\circ$ be the preimage of $B^\circ$ for the support map $M(0,2,-1)\to B$. For $1\leq k\leq d$ we define
\begin{align*}
    \Ycal_k^\circ=\left\{\Lcal\in M(0,2,-1)^\circ\;\left| \begin{array}{l}
    \Lcal(-1)\cong\Ocal_C(-\zeta_{d+1-k}+\eta_{d-k}) \mbox{ for some }C\in B^\circ,\\
    \zeta_{d+1-k}\mbox{ and }\eta_{d-k}\subset C \mbox{ have lengths equal to their indices}
    \end{array}\right.\right\}.
\end{align*}
Let $\Ycal_k=\overline{\Ycal_k^\circ}\subset M(0,2,-1)$.
\begin{prop}\label{CrossRankOneM}
Suppose $d\leq 6$. One obtains $d$ birational models of $M(0,2,-1)$ by crossing the rank one walls for $\vbf'$. These models are connected by a chain of flops
\[ \xymatrix@R-1pc@C-2pc{
& \Bl_{\Ycal_{d}}M(0,2,-1) \ar[dr]\ar[dl] && \Bl_{\tilde{\Ycal}_{d-1}}\Xcal_{d}\ar[dr]\ar[dl] && \Bl_{\tilde{\Ycal}_{k}}\Xcal_{k+1}\ar[dr]\ar[dl] &&  \Bl_{\tilde{\Ycal_1}}\Xcal_2\ar[dr]\ar[dl]\\
M(0,2,-1) \ar@{-->}[rr]^{g_{d}} && \Xcal_{d} \ar@{-->}[rr]^{g_{d-1}} && \Xcal_{d-1} \cdots\cdot \Xcal_{k+1}\ar@{-->}[rr]^{g_{k}} && \Xcal_{k}\cdots\cdot \Xcal_2\ar@{-->}[rr]^{g_1} &&\Xcal_1,} \]
where $\widetilde{\bullet}$ denotes the strict transform of the set $\bullet$ under the appropriate birational map.
\end{prop}
\begin{proof}
We consider wall-crossing along the path of stability conditions $\sigma'_t:=\sigma_{0,t}$, where $\frac{1}{\sqrt{d}}<t<+\infty$. By Lemma~\ref{MnoTSS}, there are no totally semistable walls for $\vbf'$ on this path. By Proposition~\ref{LargestWallM}, $M_{\sigma'_t}(\vbf')=M_H(\vbf')$ for $t>\sqrt{\frac{2d+1}{2d}}$. By Corollary~\ref{RankOneM} and our assumption on $d$, $\sigma'_t$ will intersect only rank one walls for $\vbf'$, given by $(1,1,k)$ with $1\leq k\leq d$ in decreasing order in $k$, as $t$ goes from $+\infty$ to $\sqrt{\frac{3}{2d}}-\delta$ for any small $\delta>0$. We will use $g_k$ to denote the flop induced by crossing the wall given by $(1,1,k)$ and $\Xcal_k$ to denote the birational model obtained from $g_k$ which corresponds to the smaller value of $t$.

The first wall $\sigma'_t$ crosses as $t$ decreases is at $t'_d=\sqrt{\frac{2d+1}{2d}}$. It induces the (unique) decomposition of Mukai vectors
\begin{align*}
    (0,2,-1)=(-1,1,-d-1)+(1,1,d).
\end{align*}
By Lemma~\ref{StayGieseker}, along with a twist by $\Ocal_S(2)$ (as Lemma~\ref{StayGieseker} is for $\sigma_{-2,t}$, but here the stability condition is $\sigma_{0,t}$), $M_{\sigma'_t}(1,1,d)=M_H(1,1,d)\cong \hilb{1}$ near $t=t_d'$. On the other hand, $M_{\sigma'_t}(-1,1,-d-1)$ is a single point representing $\Ocal_S(-1)[1]$ near $t=t'_d$. Hence a torsion sheaf $\Ecal\in M(0,2,-1)$ is in the exceptional locus of $g_d$ if and only if it fits into a non-split distinguished triangle
\begin{align*}
    \Ical_p(1)\to \Ecal\to \Ocal_S(-1)[1] ,
\end{align*}
which is equivalent to the short exact sequence of sheaves
\begin{align*}
    0\to \Ocal_S(-1)\xrightarrow{s_d}\Ical_p(1)\to \Ecal\to 0.
\end{align*}
When $s_d$ corresponds to a smooth curve in $|2H|$, the above sequence is equivalent to $\Ecal\in \Ycal_d^\circ$. Denote the exceptional locus of $g_d$ on $M(0,2,-1)$ by $E'_d$. The intersection of $E'_d$ and $\Ycal_d$ contains $\Ycal_d^\circ$. Since both $\Ycal_d$ and $E'_d$ are closed and irreducible, we must have $E'_d=\Ycal_d$.\ 

For $1\leq k\leq d-1$, $\sigma'_t$ intersects the wall given by $(1,1,k)$ at $t'_k=\sqrt{\frac{2k+1}{2d}}$. It induces the (unique) decomposition of Mukai vectors
\begin{align*}
    (0,2,-1)=(-1,1,-k-1)+(1,1,k).
\end{align*}
By Lemma~\ref{StayGieseker}, along with a twist by $\Ocal_S(2)$, near $t=t'_k$ we have $M_{\sigma'_t}(1,1,k)=M_H(1,1,k)\cong \hilb{d+1-k}$, where the last isomorphism is given by $\Ical_\eta(1)\mapsfrom\eta$. By Proposition~\ref{dual} and Lemma~\ref{StayGieseker} with a twist by $\Ocal_S(-2)$, near $t=t_k'$ we have
$M_{\sigma_t'}(-1,1,-k-1)\cong \hilb{d-k}$ via
$$\RHom(\Ical_\eta,\Ocal_S(-1))[1]\mapsfrom\eta.$$
The exceptional locus $E_k'$ of $g_k$ in $\Xcal_k$ is irreducible and a torsion sheaf $\Ecal\in E_k'$ if and only if it fits into a non-split distinguished triangle
\begin{align*}
    \Ical_{\zeta_{d+1-k}}(1)\to\Ecal\to\RHom(\Ical_{\eta_{d-k}},\Ocal_S(-1))[1],
\end{align*}
where $\zeta_{d+1-k}\in\hilb{d+1-k}$ and $\eta_{d-k}\in\hilb{d-k}$. In particular, there exists an open dense subset $E_k'^\circ\subset E_k'$ consisting of $\Ecal$ fitting into a triangle
\begin{align*}
    \RHom(\Ical_{\eta_{d-k}},\Ocal_S(-1))\xrightarrow{s_k}\Ical_{\zeta_{d+1-k}}(1)\to\Ecal,
\end{align*}
where $\eta_{d-k}\cap\zeta_{d+1-k}=\emptyset$ and $s_k$ corresponds to a curve in $B^\circ$. Assuming $\eta_{d-k}\cap\zeta_{d+1-k}=\emptyset$, the map $s_k$ is determined by a curve $C\in B^\circ$ containing $\eta_{d-k}\cup\zeta_{d+1-k}$. It is easy to see that $\Ecal(-1)\cong\Ocal_C(-\zeta_{d+1-k}+\eta_{d-k})$. Thus $E'^\circ_k\subset \Ycal_k^\circ$. In fact, it is clear that $E_k'^\circ$ is an open dense subset  of $\Ycal_k^\circ$. Since $E'_k$ is irreducible, $E'_k$ must be the strict transform of $\Ycal_k$.
\end{proof}
\begin{rem}
Similarly to Proposition~\ref{WisP^nbundle}, one can show that $\Ycal_k$ is generically a $\Pbb^{2d+2k}$-bundle over $\hilb{d-k}\times \hilb{d-k+1}$. See also \cite[Proposition 3.11]{QS22a}.
\end{rem}

\subsection{The Brill-Noether wall}
The Brill-Noether (middle) wall can be viewed as a rank one wall for both $\vbf$ and $\vbf'$. Its behavior is similar to the unique wall for the primitive Beauville-Mukai system. When viewed as a wall $\Wcal'$ for $\vbf'$, its associated hyperbolic lattice contains the vector $(1,0,1)$. Near $x=0$ this yields decompositions
\begin{align*}
    (0,2,-1)&=(1,0,1)+(-1,2,-2)\\
    &=(1,0,1)+(1,0,1)+(-2,2,-3)\qquad&\mbox{if $d\geq 2$,}\\
    &=(1,0,1)+(1,0,1)+(1,0,1)+(-3,2,-4)\qquad&\mbox{if $d\geq 3$,}\\
    &\cdots
\end{align*}
Let $\Ycal_0:=\{\Ecal\in M(0,2,-1)\;|\; h^0(S,\Ecal)\geq 1\}$ be the Brill-Noether locus. Then $\Ycal_0$ has a stratification by the dimension of the space of sections of $\Ecal$. By a similar analysis as \cite[Section 6-8]{Bay18},  $\Wcal'$ induces a stratified Mukai flop of $\Xcal_1$ at the strict transform of $\Ycal_0$.

Viewing the middle wall as a wall $\Wcal$ for $\vbf$, its associated hyperbolic lattice contains the spherical class $(-1,2,-4d+1)$. Near $x=-2$ this yields a decomposition
\begin{align*}
     (1,0,-4d)&=(-1,2,-4d-1)+(2,-2,1)\\
    &=(-1,2,-4d-1)+(-1,2,-4d-1)+(3,-4,4d+2)&\mbox{ if $d\geq 2$,}\\
    &=(-1,2,-4d-1)+(-1,2,-4d-1)+(-1,2,-4d-1)+(4,-6,8d+3)&\mbox{ if $d\geq 3$,}\\
    &\cdots
\end{align*}
Let $Y_0:=\{\xi\in \hilb{4d+1}\;|\; h^0(\Ical_\xi(2))\geq 2\}$. Then $Y_0$ has a natural stratification given by $h^0(\Ical_\xi(2))$, the dimension of the space of curves in $|2H|$ containing $\xi$. A similar analysis as above shows that  $\Wcal$ induces a stratified Mukai flop of $X_1$ at the strict transform of $Y_0$.

\section{Examples}
Using the results from the previous sections, we give complete descriptions of the birational geometry of the rank two Beauville-Mukai systems
when the degree of the surface is small. We use the notations in Section~\ref{SectExc}. Note again by Lemma~\ref{SnoTSS} that there are no totally semistable walls for $\vbf$ on the path $\sigma_{-2,t}$ where $\frac{1}{\sqrt{d}}<t<+\infty$; and by Lemma~\ref{MnoTSS}, if $\delta$ and $\epsilon$ are small enough, then there are no totally semistable walls for $\vbf'$ on the path $\sigma_{-\epsilon,t}$ where $\frac{1}{\sqrt{d}}-\delta<t<+\infty$ (note that the second path starts just below $\frac{1}{\sqrt{d}}$ so that it crosses the Brill-Noether wall).

\subsection{Degree two}
The birational geometry of the rank two Beauville-Mukai system on a surface $S$ with degree two has been investigated in \cite{Hel20}. The main theorem there can be rephrased using the notation from Section~\ref{SectExc} as follows.
\begin{prop}\cite[Theorem 1.2]{Hel20}
Let $(S,H)$ be a polarized K3 surface with $\Pic(S)=\Z \cdot H$ and $H^2=2$. There are five smooth birational models of $\hilb{5}$. They are connected by a chain of flops
\[ \xymatrix@R-1pc@C-2pc{
& \Bl_{Y_2}\hilb{5} \ar[dr]\ar[dl] && \Bl_{\tilde{Y}_1}X_{2}\ar[dr]\ar[dl]  && \mbox{       }&&\Bl_{\tilde{\Ycal_0}}\Xcal_1 \ar[dr]\ar[dl] && \Bl_{\Ycal_1}M \ar[dr]\ar[dl] \\
\hilb{5} \ar@{-->}[rr]^{f_{2}} && X_{2} \ar@{-->}[rr]^{f_1} && X_1\ar[rr]^{\Phi_2}&&\Xcal_0 &&\Xcal_1 \ar@{-->}[ll]_{g_0} &&M, \ar@{-->}[ll]_{g_{1}}} \]
where $\widetilde{\bullet}$ denotes the strict transform of the set $\bullet$ under the appropriate birational map. Here $g_0$ is the flop of $\Xcal_1$ with exceptional locus $\tilde{\Ycal_0}$ induced by crossing the Brill-Noether wall. The model $X_1$ is isomorphic to $\Xcal_0$ via $\Phi_2$.
\end{prop}
\begin{proof}
See \cite[Proof of Theorem 5.2]{Hel20}. One can also adapt the proof of the higher degree cases; see below.
\end{proof}

\subsection{Degree four}
Generic degree four K3 surfaces with Picard rank $1$ are smooth quartics in $\Pbb^3$. Since $d=2$, there are no rank two walls for either $\vbf$ or $\vbf'$ by Lemmas~\ref{Rank2S} and \ref{Rank2M}.
\begin{prop}\label{deg4}
Let $(S,H)$ be a polarized K3 surface with $\Pic(S)=\Z\cdot H$ and $H^2=4$. There are seven birational models of $\hilb{9}$. They are connected by a chain of flops
\[ \xymatrix@R-1pc@C-2pc{
& \Bl_{Y_3}\hilb{9} \ar[dr]\ar[dl]&& \Bl_{\tilde{Y}_2}X_3 \ar[dr]\ar[dl] && \Bl_{\tilde{Y}_1}X_{2}\ar[dr]\ar[dl]  && \mbox{       }&&\Bl_{\tilde{\Ycal_0}}\Xcal_1 \ar[dr]\ar[dl] && \Bl_{\tilde{\Ycal_1}}\Xcal_2 \ar[dr]\ar[dl] && \Bl_{\Ycal_2}M \ar[dr]\ar[dl] \\
\hilb{9} \ar@{-->}[rr]^{f_{3}} && X_{3} \ar@{-->}[rr]^{f_{2}} && X_{2} \ar@{-->}[rr]^{f_1} && X_1\ar[rr]^{\Phi_2}&&\Xcal_0 &&\Xcal_1 \ar@{-->}[ll]_{g_0}&&\Xcal_2 \ar@{-->}[ll]_{g_1} &&M, \ar@{-->}[ll]_{g_{2}}} \]
where $\widetilde{\bullet}$ denotes the strict transform of the set $\bullet$ under the appropriate birational map. Here $g_0$ is the stratified Mukai flop of $\Xcal_1$ with exceptional locus $\tilde{\Ycal_0}$ induced by crossing the Brill-Noether wall. The model $X_1$ is isomorphic to $\Xcal_0$ via $\Phi_2$.
\end{prop}
\begin{proof}
By Proposition~\ref{CrossRankOneS}, we obtain models $\hilb{9}$, $X_1$, $X_2$, and $X_3$ by crossing walls for $\vbf$ along $\sigma_t:=\sigma_{-2,t}$ for $\frac{1}{\sqrt{2}}<t<+\infty$ as $t$ decreases. By Proposition~\ref{CrossRankOneM}, we obtain models $M$, $\Xcal_1$, and $\Xcal_2$ by crossing walls for $\vbf'$ along $\sigma'_t:=\sigma_{0,t}$ for $\frac{1}{\sqrt{2}}<t<+\infty$. Since there are finitely many walls for $\vbf'$, the wall-crossing of the largest two walls for $\vbf'$ along $\sigma_{-\epsilon,t}$ is the same as that along $\sigma'_t$ for $\epsilon$ small enough.  Moreover, $\sigma_{-\epsilon,t}$ intersects the wall given by $(1,0,1)$ at $t_0=\sqrt{\frac{1}{2}+\frac{\epsilon}{4}-\epsilon^2}$. We use $g_0$ to denote the induced birational map as $\sigma_{-\epsilon,t}$ crosses this wall as $t$ decreases. It corresponds to the decompositions of Mukai vectors
\begin{align*}
    (0,2,-1)&=(1,0,1)+(-1,2,-2) \\
    &=(1,0,1)+(1,0,1)+(-2,2,-3).
\end{align*}
Let $\Ycal_0:=\{\Ecal\in M \;|\; h^0(S,\Ecal)\neq 0\}$ with stratification given by the dimension of the space of global sections of $\Ecal$. Then $g_0$ is the stratified Mukai flop of $\Xcal_1$ at $\tilde{\Ycal}_0$. 

It remains to show that $\Phi_2$ induces an isomorphism from $X_1$ to $\Xcal_0$. Let $\frac{1}{\sqrt{2}}<\tau<1$; then $\sigma_\tau$ is in the chamber under the wall for $\vbf$ inducing $f_1$. Arguing as in the proof of Corollary~\ref{Rank1Phi} (cf.\ also the last paragraph of the proof of \cite[Theorem 4.6]{QS22a}) and noting that $\Phi_2(\vbf)=\vbf'$, we see that $\Phi_2$ induces an isomorphism
\begin{align*}
    X_1=M_{\sigma_\tau}(\vbf)\xrightarrow{\Phi_2}M_{\Phi_2(\sigma_\tau)}(\vbf')=M_{\sigma_{x',y'}}(\vbf'),
\end{align*}
where $y'>0$, $(x',y')$ is in a chamber under the wall for $\vbf'$ inducing $g_0$, and $\sigma_{x',y'}$ is in the $\widetilde{\mathrm{GL}}_2^+(\R)$-orbit of $\Phi_2(\sigma_\tau)$. It follows that $M_{\sigma_{x',y'}}(\vbf')=\Xcal_0$. This finishes the proof.
\end{proof}

\subsection{Degrees six and eight}
Suppose $d=3$ or $4$. By Lemmas~\ref{Rank2S} and \ref{Rank2M}, there is a rank two wall for $\vbf$ given by $(2,-3,4d+2)$ and no rank two wall for $\vbf'$. 
\begin{prop}\label{deg68}
Let $(S,H)$ be a polarized K3 surface with $\Pic(S)=\Z\cdot H$ and $H^2=2d$ with $d=3$ or $4$. There are $2d+4$ birational models of $\hilb{4d+1}$. They are connected by a chain of flops
\[ \xymatrix@R-1pc@C-2pc{
& \Bl_{Y_{d+1}}S^{[4d+1]} \ar[dr]\ar[dl] && \Bl_{\tilde{Y}_d}X_{d+1}\ar[dr]\ar[dl] && && \Bl_{\tilde{Y}_1}X_2\ar[dr]\ar[dl] && \Bl_{V}X_1\ar[dr]\ar[dl] && \\
\hilb{4d+1} \ar@{-->}[rr]^{f_{d+1}} && X_{d+1} \ar@{-->}[rr]^{f_d} && X_d &\cdots & X_2 \ar@{-->}[rr]^{f_1} &&  X_1 \ar@{-->}[rr]^{h} &&X_\flat&\xrightarrow{\Phi_2}} \]
\[ \xymatrix@R-1pc@C-2pc{
&&  \Bl_{\tilde{\Ycal}_0}\Xcal_1\ar[dr]\ar[dl] && \Bl_{\tilde{\Ycal}_1}\Xcal_2\ar[dr]\ar[dl] && && \Bl_{\tilde{\Ycal}_{d-1}}\Xcal_d\ar[dr]\ar[dl] && \Bl_{\Ycal_d}M \ar[dr]\ar[dl]\\
 \xrightarrow{\Phi_2}&\Xcal_0 && \Xcal_1 \ar@{-->}[ll]_{g_0} && \Xcal_2 \ar@{-->}[ll]_{g_1}&\cdots&  \Xcal_{d-1}&&  \Xcal_d \ar@{-->}[ll]_{g_{d-1}} &&M, \ar@{-->}[ll]_{g_d}} \]
where $\widetilde{\bullet}$ denotes the strict transform of the set $\bullet$ under the appropriate birational map. Here $V\subset X_1$ is a $\Pbb^{2d+5}$-bundle over $M_H(2,-3,4d+2)\times \hilb{d-1}$ and $h$ is the flop of $X_1$ at $V$, whereas $g_0$ is the stratified Mukai flop of $\Xcal_1$ with exceptional locus $\tilde{\Ycal_0}$ induced by crossing the Brill-Noether wall. The model $X_\flat$ is isomorphic to $\Xcal_0$ via $\Phi_2$.
\end{prop}
\begin{proof}
We consider the wall-crossing for $\vbf$ along the path $\sigma_{-2,t}$ for $\frac{1}{\sqrt{d}}<t<+\infty$ as $t$ decreases. By Proposition~\ref{CrossRankOneS}, we obtain the models $\hilb{4d+1}$ and $X_{k}$ for $1\leq k \leq d+1$, and the description of the flops connecting them. By Corollary~\ref{OnleOneTwoS} and Lemma~\ref{Rank2S}, $\sigma_t$ will cross one more wall for $\vbf$, which is rank two, at $t=\frac{2}{\sqrt{3d}}$ before $t$ reaches $\frac{1}{\sqrt{d}}$. We use $h$ to denote the induced birational map as $\sigma_t$ crosses the rank two wall with $t$ decreasing and $V$ to denote the exceptional locus of $h$ in $X_1$. Note that $h$ corresponds to the decomposition of Mukai vectors
\begin{align*}
    (1,0,-4d)=(2,-3,4d+2)+(-1,3,-8d-2).
\end{align*}
By Lemma~\ref{StayGieseker}, near $t=\frac{2}{\sqrt{3d}}$ we have $M_{\sigma_t}(2,-3,4d-2)=M_H(2,-3,4d-2)$ and we have $M_{\sigma_t}(-1,3,-8d-2)\cong\hilb{d-1}$ via
$\RHom(\Ical_\eta,\Ocal_S(-3))[1]\mapsfrom \eta$. By \cite[Section 14]{BM14b}, the exceptional locus $V$ of $h$ in $X_1$ is a $\Pbb^{2d+5}$-bundle over $M_H(2,-3,4d+2)\times \hilb{d-1}$.\ 

For the second row of the diagram, we consider the wall-crossing for $\vbf'$ along the path $\sigma_{-\epsilon,t}$ for a fixed small $\epsilon>0$, and for $\frac{1}{\sqrt{d}}<t<+\infty$ as $t$ decreases. By Proposition~\ref{CrossRankOneM}, we obtain the models $M$ and $\Xcal_k$ for $1\leq k\leq d$, and the description of the flops connecting them. We use $g_{0}$ to denote the induced birational map as $\sigma_{-\epsilon,t}$ crosses the Brill-Noether wall with $t$ decreasing The description of $g_0$ and the fact that $\Phi_2$ induces an isomorphism from $X_\flat$ to $\Xcal_0$ can be checked as in the proof of Proposition~\ref{deg4}.
\end{proof}

\subsection{Degrees ten and twelve}
Suppose $d=5$ or $6$. By Lemmas~\ref{Rank2S} and \ref{Rank2M}, there is a rank two wall for $\vbf$ given by $(2,-3,4d+2)$ and for $\vbf'$ given by $(2,1,2)$. 
\begin{prop}\label{deg10}
Let $(S,H)$ be a polarized K3 surface with $\Pic(S)=\Z\cdot H$ and $H^2=2d$ with $d=5$ or $6$. There are $2d+5$ birational models of $\hilb{4d+1}$. They are connected by a chain of flops
\[ \xymatrix@R-1pc@C-2pc{
& \Bl_{Y_{d+1}}S^{[4d+1]} \ar[dr]\ar[dl] && \Bl_{\tilde{Y}_d}X_{d+1}\ar[dr]\ar[dl] && && \Bl_{\tilde{Y}_1}X_2\ar[dr]\ar[dl] && \Bl_{V}X_1\ar[dr]\ar[dl] && \\
\hilb{4d+1} \ar@{-->}[rr]^{f_{d+1}} && X_{d+1} \ar@{-->}[rr]^{f_d} && X_d &\cdots & X_2 \ar@{-->}[rr]^{f_1} &&  X_1 \ar@{-->}[rr]^{h} &&X_\flat&\xrightarrow{\Phi_2}} \]
\[ \xymatrix@R-1pc@C-2pc{
&&  \Bl_{\widetilde{\Ycal_0}}\Xcal_\flat\ar[dr]\ar[dl]&&  \Bl_{\mathcal{V}}\Xcal_1\ar[dr]\ar[dl] && \Bl_{\tilde{\Ycal}_1}\Xcal_2\ar[dr]\ar[dl] && && \Bl_{\tilde{\Ycal}_{d-1}}\Xcal_d\ar[dr]\ar[dl] && \Bl_{\Ycal_d}M \ar[dr]\ar[dl]\\
 \xrightarrow{\Phi_2}&\Xcal_0 &&\Xcal_\flat \ar@{-->}[ll]_{g_0} && \Xcal_1 \ar@{-->}[ll]_{j} && \Xcal_2 \ar@{-->}[ll]_{g_1}&\cdots&  \Xcal_{d-1}&&  \Xcal_d \ar@{-->}[ll]_{g_{d-1}} &&M, \ar@{-->}[ll]_{g_d}} \]
where $\widetilde{\bullet}$ denotes the strict transform of the set $\bullet$ under the appropriate birational map. Here $V\subset X_1$ is a $\Pbb^{2d+5}$-bundle over $M_H(2,-3,4d+2)\times \hilb{d-1}$ and $h$ is the flop of $X_1$ at $V$, whereas $\mathcal{V}\subset \Xcal_1$ is a $\Pbb^{2d+9}$-bundle over $M_H(2,1,2)\times M_H(2,1,3)$, and lastly $g_0$ is the stratified Mukai flop of $\Xcal_\flat$ with exceptional locus $\tilde{\Ycal}_{0}$ induced by crossing the Brill-Noether wall. The model $X_\flat$ is isomorphic to $\Xcal_0$ via $\Phi_2$.
\end{prop}
\begin{proof}
Adapting the proof of Proposition~\ref{deg68}, it only remains to describe the map $j$. For a fixed small $\epsilon>0$, the path $\sigma_{-\epsilon,t}$ crosses the rank one walls for $\vbf'$ as $t$ decreases. By Corollary~\ref{OnlyOneTwoM} and Lemma~\ref{Rank2M}, $\sigma_{-\epsilon,t}$ then crosses a rank two wall for $\vbf'$ at $t=\sqrt{\frac{5}{4d}+\frac{\epsilon}{2d}-\epsilon^2}$.  We use $j$ to denote the induced birational map as $\sigma_{-\epsilon,t}$ crosses the rank two wall with $t$ decreasing and $\mathcal{V}$ to denote the exceptional locus of $j$ on $\Xcal_1$. As in Lemma~\ref{StayGieseker}, one can show that $\sigma_{0,t}$ does not cross any wall for $(2,1,2)$ or $(-2,1,-3)$. Moreover, if $\epsilon$ is small enough, neither does $\sigma_{-\epsilon,t}$. Hence near $t=\sqrt{\frac{5}{4d}+\frac{\epsilon}{2d}-\epsilon^2}$ we have $M_{\sigma_{-\epsilon,t}}(2,1,2)=M_H(2,1,2)$ and $M_{\sigma_{-\epsilon,t}}(-2,1,-3)\cong M_H(2,1,3)$ via $\RHom(F,\Ocal_S)[1]\mapsfrom F$. By \cite[Section 14]{BM14b}, $\mathcal{V}$ is a $\Pbb^{2d+9}$-bundle over $M_H(2,1,2)\times M_H(2,1,3)$ and $j$ is the flop of $\Xcal_1$ at $\mathcal{V}$.
\end{proof}

\end{document}